%% file: MMS_clean.tex
\documentclass[review,hidelinks,onefignum,onetabnum]{siamart220329}

\input{pre}

\ifpdf
\hypersetup{
  pdftitle={Convergence of Planewave Approximations for Incommensurate Systems},
  pdfauthor={T. Wang, H.J. Chen, A.H. Zhou, Y.Z. Zhou and D. Massatt}
}

\begin{document}
\nolinenumbers

\maketitle

\begin{abstract}
Incommensurate structures come from stacking the single layers of low-dimensional materials on top of one another with misalignment such as a twist in orientation.
While these structures are of significant physical interest, they pose many theoretical challenges due to the loss of periodicity.
This paper studies the physical observables of continuum Schr\"{o}dinger operators for incommensurate systems.
We characterize the density of states and local density of states in real and reciprocal spaces, and develop novel numerical methods to approximate them.
In particular, we
(i) justify the thermodynamic limit of the density of states and local density of states via the operator kernel characterization;
and (ii) propose efficient numerical schemes based on planewave approximation and trapezoidal rule in reciprocal space.
We present both rigorous analysis and numerical simulations to support the reliability and efficiency of our numerical algorithms. 
\end{abstract}



\section{Introduction}
\label{sec:introduction}
\setcounter{equation}{0}

Low dimensional materials have attracted extraordinary level of interest in the materials science and physics communities due to the unique electronic, optical, and mechanical properties \cite{britnell2013strong,geim2013van,xu2013graphene}.
In particular, when two layers of 2D materials are stacked on top of each other with a small misalignment (such as a twist), they produce incommensurate moir\'{e} patterns.
In the small twist case for example, the electronic properties develop fundamental twist-dependent electronic behaviors such as Van Hove singularities and flat bands near the Fermi energy (see e.g. \cite{cao2018unconventional,carr2020electronic,carr2017twistronics,dai2016twisted,dean2013hofstadter,woods2014commensurate}).
A typical phenomenon is the appearance of strong correlations and superconductivity in twisted bilayer graphene at the magic angle of $1.1^{\circ}$ \cite{bistritzer11,cao2018correlated,cao2018unconventional}.
It is of great importance to study these structures from a theoretical and computational point of view and learn how to control the desired properties by choice of system parameters such as material type, pressure, strain, electromagnetic fields, and twist angles.

The conventional method for simulating the incommensurate systems is to construct a supercell approximation with artificial strain \cite{koda2016,komsa2013electronic,loh2015graphene}, which then allows for the use of Bloch theory and conventional band-structure methods.
However, these approaches are usually computationally expensive, as one may need extremely large supercells to achieve the required accuracy.

The purpose of this work is to study the density of states (DoS) and local density of states (LDoS) of a linear Schr\"{o}dinger operator for incommensurate systems from the mathematical perspective. Ab initio Schr\"{o}dinger models are frequently used to study incommensurate bilayer electronics \cite{cances2023, kaxiras2016} and is a natural platform for considering the non-linear density functional theory.
The DoS and LDoS characterize the spectrum distribution of the system Hamiltonian, and correspond to observables of interest in the study of 2D materials.

The first issue is that a quantitative characterization of the DoS of the Schr\"{o}dinger operator for incommensurate systems is missing.
In \cite{Avron83}, the property of the integrated DoS was studied for one-dimensional almost-periodic systems.
In \cite{massatt2017electronic}, the DoS was introduced in the weak sense within the tight-binding models.
These works on DoS can not be directly generalized to continuous models with arbitrary dimensions. 
We consider a linear Schr\"{o}dinger operator of the form $H = -\frac{1}{2}\triangle + V(x)$, then the DoS and spatial LDoS are written as 
\begin{equation*}
D(E) = \aTr \big(\delta(E-H)\big)
\qquad{\rm and}\qquad
D(E;x) = \big\langle x \big| \delta(E-H) \big| x \big\rangle
\quad {\rm for}~E\in\R~{\rm and}~x\in\R^d 
\end{equation*}
respectively, where $\delta(\cdot)$ represents the Dirac delta function and $\aTr$ represents a spatially averaged trace that will be defined in Section \ref{sec:real}.
Here, we have used the bra-ket notation and the physically motivated language $|x\rangle$, which will be replaced by rigorous mathematical formulations in Section \ref{sec:real}.
To make the objects numerically tractable, we consider the problem in the weak sense. 
More precisely, for a suitable test function $g$, we consider the DoS and spatial LDoS given by
\begin{equation}
\label{eq:dos:ldos}
\doss(g) = \aTr \big(g(H)\big)
\qquad{\rm and}\qquad
\ldos(x;g) = \big\langle x |g(H)| x \big\rangle,
\end{equation}
respectively.
Rigorous justifications of these objects are provided in Section \ref{sec:real}.
Although the linear Schr\"{o}dinger operator has a simple form, the lack of compactness, broken translation symmetry, and continuous nature of the operator make it difficult to address the above objects of an incommensurate system. 
To handle these problems, we will use the spectral theory \cite{simon2015operator} to study $g(H)$: While $g(H)$ is not a trace class operator, it can be decomposed into a discrete collection of trace class operators.

The second problem is how to efficiently evaluate the (well-defined) DoS and LDoS of an incommensurate system.
In \cite{zhou2019plane}, the Hamiltonian was discretized by a planewave basis set with a brute cutoff, and the DoS was approximated by the resulting eigenvalues.
This type of approach essentially transfers the low dimensional incommensurate problem into a high dimensional periodic problem, which is expensive most of the time and converges slowly with respect to the planewave cutoff.
In this work, we propose an efficient numerical scheme based on our formulation of DoS in reciprocal  space.
In particular, we  approximate the reciprocal LDoS within a planewave framework and then evaluate the DoS via a trapezoidal rule.
We further improve the heuristic planewave method in \cite{zhou2019plane} by introducing a novel cutoff scheme.
In particular, we split the cutoff of wave vectors in the high dimensional reciprocal space into two directions: one increases the planewave frequency while the other one traverses the reciprocal space.
A key observation is that the errors of the planewave approximations decay at completely different speeds as the cutoff increases along the two directions.
Therefore, we truncate the wave vectors in the two different directions with different cutoffs, such that the cutoff for high frequency direction can be much smaller.
We provide a rigorous numerical analysis, as well as numerical simulations of some typical incommensurate systems, to show the efficiency of our algorithms.

\vskip 0.1cm

{\bf Further remarks on existing works.}

\vskip 0.1cm

{\it On tight-binding model.}
Interest in the mathematics community has recently emerged to develop rigorous foundations, improve models, and construct computational methods for incommensurate systems.
In \cite{tritsaris2016perturbation}, a general methodology based on perturbation theory was proposed for simulating the weakly coupled two-dimensional layers. 
In \cite{cances2017generalized}, the Kubo formula for the transport properties of incommensurate systems was given by exploiting the $C^*$-algebra approach.
In \cite{massatt2017electronic}, the thermodynamic limit of the DoS was justified within the tight-binding models and represented by an integral over local configurations. 
In \cite{massatt2018incommensurate}, the DoS of the tight-binding models was characterized in the reciprocal space by using a Bloch transformation operator. 
In \cite{etter2020modeling,massatt2020efficient}, efficient numerical algorithms were designed for calculating the conductivity of incommensurate systems.
In \cite{massatt2021electronic}, the relations between the real space, configuration space, and reciprocal space for incommensurate two-dimensional systems were classified, which provides a general class of electronic observables with a mathematical foundation.
Most of the existing works focus on the tight binding models, while the studies of continuous electronic structure models are less common (see recent works \cite{cances2022simple,watson2022bistritzer,zhou2019plane}) due to the heavy computation cost.

\vskip 0.1cm

{\it On quasi-periodic problems.}
The incommensurate system is in fact a typical quasi-periodic system \cite{dinaburg1975one,karpeshina2013multiscale}.
The Schr\"{o}dinger operator for quasi-periodic systems has attracted much research interest and there are many results concerning the spectral properties in the literature.
In \cite{Avron82,Avron83,dinaburg1975one,eliasson1992floquet,surace1990schrodinger}, the spectral properties of one-dimensional quasi-periodic systems were studied thoroughly for both discrete and continuous setting.
In \cite{karpeshina2013multiscale,karpeshina2019extended}, the existence of absolutely continuous spectrum at high energy in the two-dimensional case was discussed. 
Nevertheless, the case of general dimension is significantly more complicated and is still an open problem. 
Most of the existing works focus on theoretical characterization of the spectrum set, while the study towards the related physical observables is still missing.
In this work, the physical observables are viewed as the acting of a smooth test function on the spectral density, and their thermodynamic limit are justified in a ``weak" sense.
The theory works for arbitrary dimensions and allows us to further design numerical schemes to estimate these quantities.

\vskip 0.1cm

{\bf Outline.}
The rest of this paper is organized as follows. 
In Section \ref{sec:incommensurate}, the incommensurate systems, the DoS, and the LDoS of related Schr\"{o}dinger operators are briefly introduced.
In Section \ref{sec:real}, the DoS and LDoS of the incommensurate systems are justified with the help of a planewave description of the operator kernel.
In Section \ref{sec:pw}, some efficient numerical schemes are proposed based on the planewave discretization and the reciprocal space quadrature rule. 
In Section \ref{sec:comput}, the numerical experiments on some incommensurate systems are performed to support the theory. 
In Section \ref{sec:conclusions}, some conclusions are drawn. 
The detailed proofs are put in the appendices.

\vskip 0.1cm
{\bf Notations.}
In this paper, we will denote the dimension of the systems by $d$, and $B_R(x)\subset\R^d$ the ball centered at $x$ with radius $R$.
In particular, $B_R$ will denote the ball centered at the origin.
For a bounded domain $\Omega\subset\R^d$, $|\Omega|$ will denote its volume. 
For a finite discrete set $\mathcal{I}$, $\#\mathcal{I}$ will denote its cardinality.  
The Schwartz space, the set of all rapidly decreasing smooth functions, will be denoted by $\Sc(\R^d)$ with
\begin{equation*}
\label{def:schwartzspace}
\Sc(\R^d)=\Bigg\{f\in C^{\infty}(\R^d) :~ \sup_{x\in \R^d}(1+|x|^2)^{t} |\partial^{\alpha}f(x)|<\infty \quad {\rm for~all}~t>0,~\alpha\in \N^d \Bigg\} .
\end{equation*}
We will denote by $|\cdot|$ the Euclidean norm of a vector, and by $||\cdot||_{L^p}~(1\leq p<\infty)$ the $L^p$ norm of a function.
For a trace class operator $A$, we will denote its trace by ${\rm Tr}(A)$ and its trace norm by $\|A\|_{\rm tr}$.
For a self-adjoint operator $A:\mathcal{D} \subset L^2(\R^d)\rightarrow L^2(\R^d)$, we will denote its operator kernel by $K_A(x,y)$, such that
$A\phi(x)=\int_{\R^d}K_A(x,y)\phi(y)\dd y$ for $\phi\in L^2(\R^d)$.
We will denote the Fourier transform by $\FT: L^2(\R^d)\rightarrow L^2(\R^d)$ and the inverse Fourier transform by $\FT^{-1}: L^2(\R^d)\rightarrow L^2(\R^d)$, which are given on $\Sc(\R^d)$ by
\begin{align}
\label{trans:F}
\FT \psi(\xi) := \hat\psi(\xi) 
:= \int_{\R^d}e^{-i\xi\cdot x}\psi(x) \dd x
\qquad{\rm and}\qquad
\FT^{-1}\hat \psi(x) := \frac{1}{(2\pi)^{d}}\int_{\R^d}e^{i\xi\cdot x} \hat\psi(\xi)\dd\xi ,
\end{align}
respectively.
Throughout this paper, the symbol $C$ will denote a generic positive constant that may change from one line to the next, which will always remain independent of the approximation parameters and the choice of test functions.
The dependencies of $C$ will normally be clear from the context or stated explicitly.

\section{Schr\"{o}dinger operator for incommensurate systems}
\label{sec:incommensurate}
\setcounter{equation}{0}

We consider two $d$-dimensional $(d=1,2)$ periodic systems that are stacked in parallel along the $(d+1)$th dimension. 
We will neglect the $(d+1)$th dimension and the distance between the two layers for simplicity of the presentations,
since this dimension raises no essential mathematical insights of the problem\footnote{The full $(d+1)$-dimensional space is a product of the $d$-dimensional subspace in which the two periodic lattices lie in and the direction along which the systems are stacked.
All the difficulties of mathematical understanding come from that the system is extended and non-periodic in the $d$ dimensional subspace, while in the $(d+1)$th direction one can easily construct a basis set for the system.

Nevertheless, we should point out that in practical simulations and various modeling works, the inter-layer displacement in the $(d+1)$th direction plays an very important role \cite{bistritzer11}.}.

Each individual periodic layer can be described by a Bravais lattice
\begin{eqnarray*}
\RL_j := \big\{ A_j n~:~ n\in\Z^d \big \} \quad{\rm for}~ j=1,2,
\end{eqnarray*}
where $A_j\in\R^{d\times d}~(j=1,2)$ is invertible. 
The periodicity implies translation invariant with respect to its lattice vectors, i.e. $\RL_j=A_j n+\RL_j,~\forall~n\in\Z^d$.
The unit cell for the $j$-th layer is denotedy by
\begin{eqnarray*}
\Gamma_j := \big\{ A_j \alpha ~:~ \alpha\in[0,1)^d \big\} \quad {\rm for}~ j=1,2 .
\end{eqnarray*}
The associated reciprocal lattice and reciprocal unit cell are then given by 
\begin{eqnarray*}
\RL^*_j = \big\{ 2\pi A_j^{-{\rm T}} n~:~ n\in\Z^d \big \} \qquad {\rm and} \qquad 
\Gamma^*_j = \big\{ 2\pi A_j^{-{\rm T}}  \alpha ~:~ \alpha\in[0,1)^d \big\} ,
\end{eqnarray*}
respectively. 
Although each individual lattice $\RL_j$ is periodic, 
the joined system $\RL_1\cup\RL_2$ may lose the translation invariance property.
We give the following definition of {\it incommensurate} for systems consisting of two periodic lattices.
 
\begin{definition}
\label{incomm}
Two Bravais lattices $\RL_1$ and $\RL_2$ are called incommensurate if
\begin{eqnarray*}
\RL_1\cup\RL_2 + \tau = \RL_1\cup\RL_2 \quad \Leftrightarrow \quad \tau=\pmb{0}\in\R^d .
\end{eqnarray*}
\end{definition}

In this paper, we will consider systems such that not only the lattices $\RL_1$ and $\RL_2$ are incommensurate, but their associated reciprocal lattices $\RL^*_1$ and $\RL^*_2$ are also incommensurate\footnote{We provide an example in the following to indicate that the incommensurate conditions in real and reciprocal spaces are not equivalent.
Let
\[
\RL_1 = \left[\begin{array}{cc}
     1 & \sqrt{2} \\
     0 & \sqrt{3} 
\end{array}\right]\Z^2
\qquad{\rm and}\qquad
\RL_2 = \left[\begin{array}{cc}
     1 & 0 \\
     0 & 1 
\end{array}\right]\Z^2 .
\]
Then we have
\[
\RL^*_1 = 2\pi\left[\begin{array}{cc}
     1 & 0 \\
     -\sqrt{6}/3 & \sqrt{3}/3 
\end{array}\right]\Z^2
\qquad{\rm and}\qquad
\RL^*_2 = 2\pi\left[\begin{array}{cc}
     1 & 0 \\
     0 & 1 
\end{array}\right]\Z^2 .
\]
We can see that $\RL_1$ and $\RL_2$ are commensurate (along the $(1,0)^{\rm T}$ direction), while $\RL^*_1$ and $\RL^*_2$ are incommensurate.}. 
We focus on the following Schr\"{o}dinger operator for a bi-layer incommensurate system
\begin{eqnarray}
\label{H}
H := -\frac{1}{2}\Delta + v_1 + v_2,
\end{eqnarray}
where $v_j:\R^d\rightarrow\R~(j=1,2)$ are smooth and $\RL_j$-periodic functions. 
Due to the periodicity of the potentials, we can write $v_j$ by the Fourier series
\begin{eqnarray}
\label{V_series}
v_j(x) = \sum_{G\in \RL_j^*} \hat{V}_{j,G} e^{iG\cdot {\bf x}} \qquad{\rm with}\qquad 
\hat{V}_{j,G} = \frac{1}{|\Gamma_j|}\int_{\Gamma_j}v_j(x)e^{-iG\cdot x}\dd x \qquad j=1,2.
\end{eqnarray}
We will assume that the potentials $v_j~(j=1,2)$ are analytic functions, then the Fourier coefficients can decay exponentially fast, i.e. 
\begin{eqnarray}
\label{decayVj}
|\hat{V}_{j,G}| \leq C e^{-\gamma |G|} 
\qquad {\rm for}~G\in\RL_j,~j=1,2
\end{eqnarray}
with some $\gamma>0$.
Note that the analytic assumption on $v_j$ is actually very strong, but this is for simplicity and clarity of the presentations so that the planewave approximations in Section \ref{sec:pw} can process a clean exponential convergence rate.
Less regularity of the potentials $v_j$ will lead to corresponding slower convergence rate (for example, $v_j\in C^{\infty}$ will lead to super-algebraic convergence rate). Other than this, all theoretical results can be generalized directly without difficulty.
By using \eqref{H}, \eqref{V_series} and the definition of Fourier transform in \eqref{trans:F}, we have that
\begin{equation}
\label{FTH0}
\Big(\FT(H\psi)\Big)(\xi) = \frac{1}{2}|\xi|^2 \hat \psi(\xi) + \sum_{G_1 \in \RL_1^*} \hat V_{1,G} \hat \psi(\xi - G_1) + \sum_{G_2 \in \RL_2^*} \hat V_{2,G} \hat \psi(\xi - G_2)
\qquad{\rm for}~\xi\in\R^d.
\end{equation}

The physical observables of the system are determined by the DoS of the Hamiltonian operator \eqref{H}. 
More precisely, it is given by the ``trace" of $g(H)$, where $g\in\Sc(\R)$ is related to the observables under consideration.
In this work, we consider functions that possess stronger regularities, which belong to the following space
\begin{multline}
\label{deftestfunctiong}
\qquad
\setg:=\Big\{g\in \Sc(\R) ~:~ g~{\rm admits~an~analytic~continuation~to}
\\[1ex]
~ S_{\delta} = \{z\in \C,~|{\rm Im} z| \leq \delta\} ,
~~{\rm and}~~ \big|g(z)\big|\leq Ce^{-\zeta |{\rm Re}z|}~~\forall~z\in S_{\delta} \Big\}
\qquad
\end{multline}
with some $\zeta>0$ and $\delta>0$.
A frequently used example is to consider a Gaussian centred at $E\in\R$
\begin{align*}
g(\lambda) := \frac{1}{\sqrt{2\pi}\eps}e^{-(\lambda-E)^2/2\eps^2} 
\qquad\lambda\in\R ,
\end{align*}
and then approximate the DoS by convoluting the spectral density measure with the Gaussian.
Here the width of Gaussian $\eps$ can be thought of as a DoS resolution, which typically should be small to capture detailed spectral features such as Van Hove singularities and flat bands.
From the physical viewpoint of the DoS, we are interested in an ``averaged trace" of $g(H)$ as we are considering extended systems over $\R^d$.
However, neither the operator $g(H)$ nor the ``averaged trace" is well defined yet.
We will first show in the next section that the trace averaged to unit volume can be justified as the thermodynamic limit.

At the end of this section, we introduce a ``shifted" Hamiltonian in the reciprocal space, which will be crucial for both proving the existence and providing a representation of LDoS (see Section \ref{sec:real}). 
It will also give the structural background for the numerical algorithm proposed in this work (see Section \ref{sec:pw}). 
For $\xi\in\R^d$, let $\hH(\xi) : \C^{\RL_1^* \times \RL_2^*} \rightarrow \C^{\RL_1^* \times \RL_2^*}$ be given by
\begin{align}
\label{Hhat:xi}
\big[\hH(\xi)\big]_{\bG,\bG'} = \frac{1}{2} \big|\xi+G_1+G_2\big|^2\delta_{G_1G_1^{'}} \delta_{G_2G_2^{'}} + \hat{V}_{1,G_1-G_1^{'}} \delta_{G_2G_2^{'}} + \hat{V}_{2,G_2-G_2^{'}} \delta_{G_1G_1^{'}} 
\end{align}
with ${\bf G} = (G_1,G_2), ~\bG'=(G_1^{'},G_2^{'}) \in \RL_1^* \times \RL_2^*$.
We can see from a rigorous justification (see Appendix \ref{sec:proof:lemma:gH00}, Lemma \ref{lemma:locality:reciprocal}) that for $g\in \setg$, the matrix elements of $g\big(\hH(\xi)\big):\C^{\RL_1^* \times \RL_2^*} \rightarrow \C^{\RL_1^* \times \RL_2^*}$ are well-defined as thermodynamic limits in the reciprocal space.

\begin{remark}[connection between $\hH(\xi)$ and $H$]
\label{remark:Hxi}
For $\xi\in\R^d$, let $\unfold_\xi : \Sc(\R^d) \rightarrow \C^{\RL_1^* \times \RL_2^*}$ be an ``unfolding" operator in the reciprocal space defined by $(\unfold_\xi \hat{\psi})({\bf G}) = \hat{\psi}\big(\xi + G_1 + G_2\big)$ for $\hat{\psi}\in\Sc(\R^d)$ and ${\bf G} = (G_1,G_2) \in \RL_1^* \times \RL_2^*$.
The definition of $\unfold_\xi$ together with \eqref{FTH0} and \eqref{Hhat:xi} implies 
\begin{eqnarray}
\label{connection}
\unfold_\xi \big(\widehat{H\psi}\big) = \hH(\xi) \big(\unfold_\xi\hat\psi\big)
\qquad \forall~\psi\in \Sc(\R^d) .
\end{eqnarray}
We refer to Appendix \ref{sec:proof:lemma:gH00} (proof of \eqref{connection}) for a detailed proof of this connection.
\end{remark}

\section{Thermodynamic limit of the density of states}
\label{sec:real}
\setcounter{equation}{0}

We first justify the trace of a real-space operator $g(H)$~($g \in \setg$) that is restricted on a bounded domain.
To do this, we introduce a partition of unity on $\R^d$.
Let $\chi\in C^{\infty}(\R^d)$ be a cut-off function satisfying
\begin{equation*}
\label{partitionofunity}
\chi \geq 0, \qquad
\chi(x) = \left\{  
\begin{array}{lr}  
1,  \quad |x|\leq\frac{1}{3}
\\[1ex]  
0,  \quad |x|>1
\end{array}
\right. 
\qquad {\rm and} \qquad
\sum_{j\in \Z^d}\chi(x-j) \equiv 1 .
\end{equation*}
Let $\chi_j(x):=\chi(x-j)$. 
Then $\{\chi_j\}_{j\in \Z^d}$ forms a smooth partition of unity of $\R^d$.
In the following, $\chi_j~(j\in \Z^d)$ can also be viewed as multiplication operators on $L^2(\mathbb{R}^d)$ based on the context. 
Using the spectral theory \cite{simon2015operator}, we can derive the following lemma, which justifies the trace of the real-space ``restricted" operator $\chi_j g(H) \chi_k$ and provides a planewave representation of the operator kernel.
The proof is given in Appendix \ref{sec:proof:lemma:chigH}.

\begin{lemma}
\label{lem:boundtracejk}
Let $j,k\in \Z^d$ and $g \in \setg$. 
Then the operator $\chi_j g(H)\chi_k$ is trace class, and it has a continuous kernel $K_{jk}(x,y)$ with 
\begin{equation}
\label{K_jk}
K_{jk}(x,y) = \frac{1}{(2\pi)^d}  \chi_j(x)\chi_k(y)\sum_{{\bf G}=(G_1,G_2) \in \RL_1^*\times \RL_2^*}e^{-i (G_1+G_2)\cdot y} \int_{\R^d}e^{i\xi\cdot(x-y)} \big[g(\hH(\xi))\big]_{{\bf 0},{\bf G}} \;\dd\xi .
\end{equation}
Moreover, there exists a constant $C>0$ independent of $j$ and $k$ such that
\begin{equation}
\label{lem:trjk}
{\rm Tr}\big(\chi_jg(H)\chi_k\big) = \int_{\R^d} K_{jk}(x,x)\dd x
~ \left\{
\begin{array}{ll}
\leq C & \quad{\rm if}~|j-k|\leq 1 
\\[1ex]
= 0 & \quad{\rm if}~|j-k|>1
\end{array} 
\right. .
\end{equation}
\end{lemma}

To derive the thermodynamic limit of DoS, we consider the trace of $g(H)$ restricted on a bounded domain $B_R$, divided by the volume $|B_R|$, and then take the limit of $R\rightarrow\infty$.
For $g\in\setg$ and a given $R>0$, we define the ``averaged" trace on $B_R$ as
\begin{equation}
\label{def:dosR}
\aTrR\big(g(H)\big):=\frac{1}{|B_R|}\sum_{j,k\in \Z^d\cap B_R}{\rm Tr}\big(\chi_jg(H)\chi_k\big) . 
\end{equation}
We have from Lemma \ref{lem:boundtracejk} that $\aTrR\big(g(H)\big)$ can be uniformly bounded by a constant independent of $R$.
We shall further show in the following theorem that the limit $\lim_{R\rightarrow\infty}\aTrR\big(g(H)\big)$ exists, which will verify that the limit of \eqref{def:dosR} provides a well-defined DoS.
Moreover, the limit can also be represented by the LDoS in reciprocal space
\begin{align}
\label{ldos:reciprocal}
\widehat{\ldos}(\xi;g) := \big[g(\hH(\xi))\big]_{{\bf 0,0}} 
\qquad{\rm for}~\xi\in\R^d.
\end{align}
The proof of the theorem is given in Appendix \ref{sec:proof:theo:tdllimit}.

\begin{theorem}
\label{theo:tdllimit}
Let $g\in \setg$.
Then the limit $\displaystyle \aTr\big(g(H)\big):=\lim_{R\rightarrow\infty} \aTrR\big(g(H)\big)$ exists,
and
\begin{align}
\label{tdl}
\aTr\big(g(H)\big)
= \int_{\R^d} \widehat{\ldos}(\xi;g) \dd\xi .
\end{align}
\end{theorem} 

Theorem \ref{theo:tdllimit} elucidates that the thermodynamic limit of $\aTrR\big(g(H)\big)$ is well-defined in the weak sense.
It gives us the physical observable (averaged on unit volume) for incommensurate systems. 

We then introduce the notion of the spatial LDoS, which formally is of the form $\langle x| g(H) |x\rangle$. 
We make this definition precise through the operator kernel.
Based on Lemma \ref{lem:boundtracejk} and the definition of the partition of unity, we can give the kernel of $g(H)$ by $\sum_{j,k\in\Z^d} K_{jk}(x,y)$.
Then for $x\in\R^d$ and $g\in\setg$, the spatial LDoS is defined as the diagonal of the kernel
\begin{equation}
\label{ldos}
\ldos(x;g) := \sum_{j,k\in\Z^d} K_{jk}(x,x) 
= \frac{1}{(2\pi)^d} \sum_{{\bf G} \in \RL_1^*\times \RL_2^*}e^{-i (G_1+G_2)\cdot x} \int_{\R^d} \big[g(\hH(\xi))\big]_{{\bf 0},{\bf G}} \dd\xi.
\end{equation}
The spatial LDoS is an observable used for studying applications such as electron density and local states in Moir\'{e} patterns \cite{simon2015operator}.
Since the lattices $\RL_1$ and $\RL_2$ are incommensurate, we observe that any $x \in \mathbb{R}^d$ can be described by a pair $(b_1,b_2) \in \Gamma_1\times \Gamma_2$, where $b_j = x + R_j \in \Gamma_j$ with a correct choice of $R_j \in \mathcal{R}_j$ (see Figure \ref{fig:config}, or \cite{cances2017generalized}).
We further provide an alternative description of the LDoS, with respect to the ``local" configuration $(b_1,b_2)\in \Gamma_1\times \Gamma_2$ as
\begin{equation}
\label{ldos:b1b2}
\widetilde{\ldos}(b_1,b_2;g) := \sum_{{\bf G} \in \RL_1^*\times \RL_2^*}e^{-i (b_1\cdot G_1+b_2 \cdot G_2)} 
\int_{\R^d} \big[g(\hH(\xi))\big]_{{\bf 0},{\bf G}} \dd\xi.
\end{equation}
We can trivially compute from the LDoS definitions that if $b_1(x)$ and $b_2(x)$ are the modulation of $x$ with respect to $\Gamma_1$ and $\Gamma_2$ respectively, then
\begin{equation*}
\ldos(x;g) = \widetilde{\ldos}(b_1(x),b_2(x),g).
\end{equation*}
Integration immediately yields the density of states formula in terms of LDoS:
\begin{align}
\label{tdl:b1b2}
\aTr\big(g(H)\big) = \frac{1}{|\Gamma_1||\Gamma_2|} \int_{\Gamma_1}\int_{\Gamma_2} \widetilde{\ldos}(b_1,b_2;g) \dd b_1\dd b_2.
\end{align}

\begin{figure}[!htb]
\centering
\includegraphics[width=7.5cm]{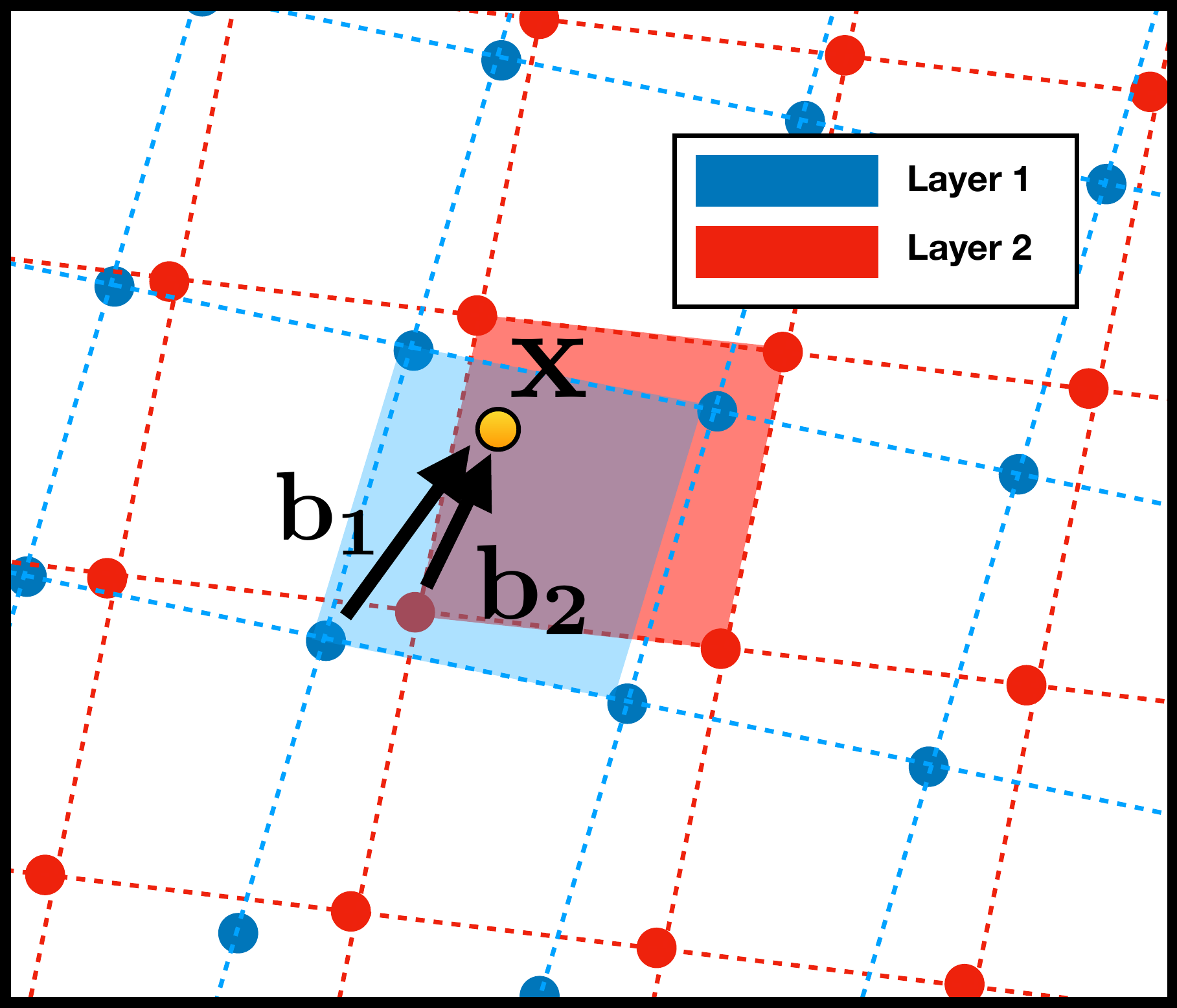}
\caption{The unit cells $\Gamma_j$ are shaded and their corresponding lattices $\mathcal{R}_j$ are displayed. Configuration vectors $b_1$ and $b_2$ corresponding to $x \in \mathbb{R}^2$ (central dot in both unit cells).}
\label{fig:config}
\end{figure}

The characterization of the DoS in Theorem \ref{theo:tdllimit} and \eqref{tdl:b1b2} can be viewed as a generalization of the results derived for incommensurate tight-binding models \cite{massatt2021electronic,massatt2018incommensurate,massatt2017electronic} to continuum models.
We mention that all the results can be extended to incommensurate systems with more than two layers. 
In particular, for a system with $M$ layers, one shall describe the $\hH(\xi)$ matrix in \eqref{Hhat:xi} with an $M$-dimensional reciprocal lattice and uses an $M$-fold multiple integral in the configuration representation \eqref{tdl:b1b2}.

\section{Planewave approximations and convergence analysis}
\label{sec:pw}
\setcounter{equation}{0}

In this section, we will propose two planewave methods to approximate the DoS of the incommensurate systems.
The first method is a natural approximation of the reciprocal space formula given in the previous section, with an efficient planewave cutoff and a trapezoidal rule in the reciprocal space.
We provide a rigorous analysis for the convergence rates with respect to the planewave cutoff parameters and the quadrature mesh size, which supports the efficiency of the numerical schemes.
The second method is a direct planewave discretization of the continuous Hamiltonian \eqref{H}, without any ``shifting" as in \eqref{Hhat:xi}.
Though generally less efficient than our first method in practice, it provides a significant improvement of the previous works with clever cutoff of the planewave vectors.
We also provide rigorous analysis for the convergence rates with the cutoff parameters.

\subsection{Planewave discretization based on LDoS in reciprocal space}
\label{sec:pw:discretization}

We have shown in Theorem \ref{theo:tdllimit} that the DoS is given by the integral of the LDoS in reciprocal space.
Then the approximation of DoS consists of two steps: 
(a) the planewave approximation of the LDoS in reciprocal space, i.e. $\big[g\big(\hH(\xi)\big)\big]_{\pmb{0},\pmb{0}}$; and 
(b) the numerical quadrature with respect to $\xi$ over $\R^d$.

We first consider the approximation of of $\big[g\big(\hH(\xi)\big)\big]_{\pmb{0},\pmb{0}}$.
Let $\Omega\subset\RL_1^* \times \RL_2^*$ give a finite dimensional truncation of the planewave vectors, we denote by $\hH^{\Omega}(\xi):\ell^2(\Omega)\to\ell^2(\Omega)$ the Hamiltonian with the same matrix elements given in \eqref{Hhat:xi}, but with the wave vectors $\bG,\bG'$ restricted on $\Omega$.
To get a concrete finite dimensional planewave approximation of $[g(\hH(\xi))]_{{\bf 0,0}}$, we introduce the following truncation for the planewave vectors in $\RL_1^* \times \RL_2^*$.
Let $W,~L >0$ and
\begin{eqnarray}
\label{set:cutoff}
\DD_{W,L} := \Big\{ \big(G_{1},G_{2}\big) \in\RL_1^* \times \RL_2^* ~:~ \big|G_{1}+G_{2}\big|\leq W ,~\big|G_{1}-G_{2}\big|\leq L \Big\}.
\end{eqnarray}
The degrees of freedom (i.e. the number of planewaves) in $\DD_{W,L}$ scales like $\mathcal{O}\big(L^d W^d\big)$.
Then the LDoS in reciprocal space $\big[g\big(\hH(\xi)\big)\big]_{{\bf 0,0}}$ will be approximated by $\big[g\big( \hat{H}(\xi)^{\DD_{W,L}} \big)\big]_{\bzero,\bzero}$ within our framework.
The reason we design the cutoff bounded by $W$ and $L$ is that the decays of the approximate errors could be significantly different with respect to the two parameters, which we will see in both theory (Section \ref{sec:pw:convergence}) and numerics (Section \ref{sec:comput}).
We give a schematic plot of the cutoffs for one dimensional systems (i.e. $d=1$) in Figure \ref{fig:D_WL}, in which the domain $\DD_{W,L}$ has very different widths.

\begin{figure}[!htb]
\centering
\includegraphics[width=9.5cm]{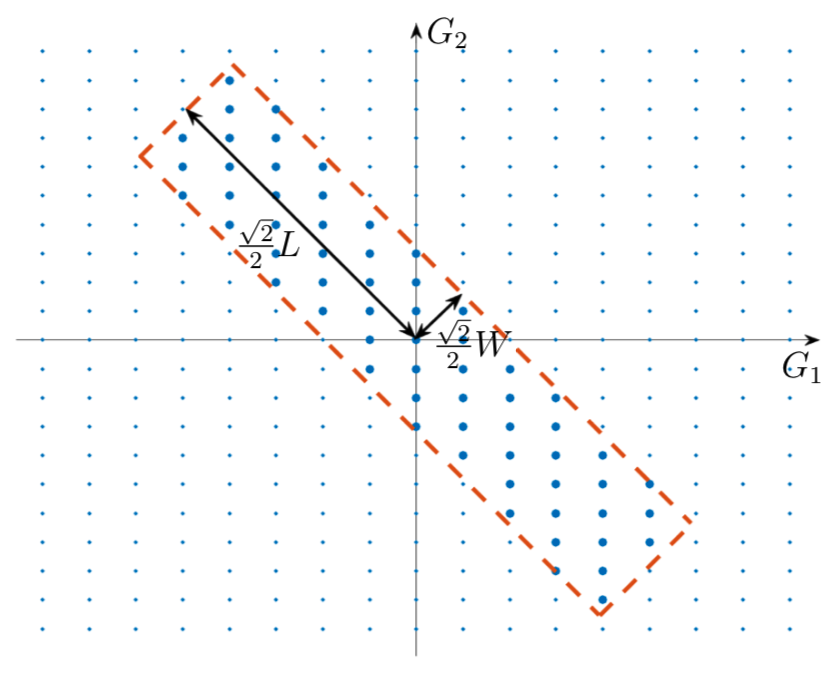}
\caption{Schematic plot of the domain $\DD_{W,L}$ in the reciprocal space.}
\label{fig:D_WL}
\end{figure}

In practical calculations, one can evaluate $\big[g\big(\hat{H}(\xi)^{\DD_{W,L}}\big)\big]_{\bzero,\bzero}$ directly by using the formula
\begin{align}
\label{ldos:eigen}
\big[g\big(\hat{H}(\xi)^{\DD_{W,L}}\big)\big]_{\bzero,\bzero} = \sum_{j} g(\lambda_j^{\xi})\big[\psi_j^{\xi}\big]_{\bzero}^2, 
\end{align}
where $(\lambda_j^{\xi}, \psi_j^{\xi})$ are eigenpairs of the matrix $\hat{H}(\xi)^{\DD_{W,L}}$.
Alternatively, one can approximate the matrix elements of $g\big(\hat{H}(\xi)^{\DD_{W,L}}\big)$ by the kernel polynomial methods, see e.g. \cite{massatt2018incommensurate,massatt2017electronic,silver1996kernel,weisse2006kernel}.

We next consider the integral of $\big[g\big(\hat{H}(\xi)^{\DD_{W,L}}\big)\big]_{\bzero,\bzero}$ with respect to $\xi$ in \eqref{tdl}.
Let $K\in\N^+$ and $h=W/K$.
We first construct a uniform quadrature mesh $\Kc_{h}$ on $[-W,W]^d$ with
\begin{equation}
\label{domainofsampling}
\Kc_{h}^W = \Big\{\big(x_{k_1}, \cdots, x_{k_d} \big) ~:~ -K\leq k_j < K,~1\leq j\leq d \Big\}
\quad{\rm and}~ x_j=jh.
\end{equation}
We then approximate the DoS defined in \eqref{tdl} by 
\begin{equation}
\label{discreteddos}
\DosK{W,L}{\DD_{W,L}}: = h^d \sum_{\xi\in \Kc_h^W} g\big( \hat{H}(\xi)^{\DD_{W,L}} \big)_{\bzero,\bzero}.
\end{equation}

The cost of the approximation \eqref{discreteddos} scales like $\mathcal{O}\big(K^d(LW)^{3d}\big)$ if \eqref{ldos:eigen} is used to evaluate the LDoS in reciprocal space, where the cubic scaling comes from the diagonalization of the matrices $\hat{H}(\xi)^{\DD_{W,L}}$.

The following theorem gives the convergence of the approximation \eqref{discreteddos}, which shows that the numerical errors decay exponentially fast with respect to the planewave cutoff and the quadrature mesh size. 
There is a numerically critical asymptotic difference between the $L$ and $W$ plane-wave truncation parameters when balancing errors, which is determined by the properties of $g$. 
The $L$-truncation has a convergence rate dependent on $g$'s regularity, and can be thought of as an ergodic sampling of wavenumbers in the $L\rightarrow \infty$ limit. 
On the other hand, the $W$-truncation corresponds to wavenumber cut-off. The Laplacian term sends large wavenumbers to high energy, and via this mechanism exponential localization of $g$ in energy yields an exponential convergence with respect to the $W$-truncation. 
Thus, as we will see in Section \ref{sec:comput} the $L$-truncation will typically need to be orders of magnitude larger than the $W$-truncation.
The proof of this theorem is given in Appendix \ref{sec:prooftheo:convergence_discretizedDos}.

\begin{theorem}
\label{theo:convergence_discretizedDos}
Let $g\in\setg$, then there exist positive constants $C$ and $c$ that do not depend on $L$, $W$, $h$ and $g$, such that
\begin{equation}
\label{error-dos-sampling}
\left\vert 
\DosK{W,L}{\DD_{W,L}} - \aTr\big(g(H)\big)
\right\vert 
\leq C
\big( \delta^{-2} e^{- c\delta L} + \delta^{-2} e^{- c \zeta W} + e^{- c\delta/h} \big).
\end{equation}
\end{theorem}

\begin{remark}
\label{remark:h}
We mention that the exponential decay with respect to the quadrature mesh size $h$ comes from the trapezoidal rule for analytic functions (see the proof in Appendix \ref{sec:prooftheo:convergence_discretizedDos}).
This requires an integration over the whole space $\R^d$, and can be clearly observed when the cutoff $W$ is sufficiently large (see numerical results in Section \ref{sec:comput}).
However, when the cutoff $W$ is fixed at a small value, one will observe an $\mathcal{O}(h^2)$ convergence with respect to the mesh size since the integrand loses the analyticity without the tails outside $[-W,W]^d$ (see numerical results in Section \ref{sec:comput}).
In practice, we will need to balance the errors from the cutoff $W$ and the quadrature mesh size $h$.
\end{remark}

\subsection{Planewave discretization of the Hamiltonian without shifting}
\label{sec:pw:convergence}

In this subsection, we will consider an alternative planewave approximation of the DoS, which is more of a direct planewave discretization of the continuous problem \eqref{H}.
This type of construction was first proposed in \cite{zhou2019plane}, see also earlier works on quasiperiodic problems \cite{jiang2014numerical}.
We will show that such a direct approach can also converge to the thermodynamic limit of DoS derived in Theorem \ref{theo:tdllimit}, which will provide much more efficient numerical schemes than previous works with clever cutoffs of the planewave vectors.

Define the discrete Hamiltonian $\hH:\C^{\RL_1^* \times \RL_2^*} \rightarrow \C^{\RL_1^* \times \RL_2^*}$ in the reciprocal space 
such that
\begin{equation}
\label{Hpw}
\hH_{\bG,\bG'} = \frac{1}{2}|G_1+G_2|^2\delta_{G_1G_1^{'}} \delta_{G_2G_2^{'}} + \hat{V}_{1,G_1-G_1^{'}} \delta_{G_2G_2^{'}} + \hat{V}_{2,G_2-G_2^{'}} \delta_{G_1G_1^{'}} 
\end{equation}
for $\bG=(G_1,G_2),~\bG'=(G_1^{'},G_2^{'}) \in \RL_1^* \times \RL_2^*$.
Note that $\hH=\hH(\bzero)$ with $\hH(\xi)$ defined in \eqref{Hhat:xi}.
With the Hamiltonian in \eqref{Hpw} and the cutoff scheme given by \eqref{set:cutoff}, the DoS given in \eqref{tdl} can be approximated by
\begin{align}
\label{def:dos-reciprocalWL}
\pwdosWL{\hH^{\DD_{W,L}}} :=\frac{|\Gamma_1^*||\Gamma_2^*|}{S_{d,L}}\Tr\Big(g\big(\hH^{\DD_{W,L}}\big)\Big),
\end{align}
where $S_{d,L}$ denotes the volume of a $d$-dimensional ball with diameter $L$.
Here, the pre-factor $|\Gamma_1^*||\Gamma_2^*|/S_{d,L}$ is needed such that the DoS can be correctly ``averaged" to unit volume.

In practice, the trace $\Tr\big(g(\hH^{\DD_{W,L}})\big)$ on the right-hand side of \eqref{def:dos-reciprocalWL} can be evaluated once the finite dimensional Hamiltonian $\hH^{\DD_{W,L}}$ is obtained.
One can either solve the matrix eigenvalue problem to get the eigenvalues $\{\lambda_j\}$ of $\hH^{\DD_{W,L}}$ and take the sum $\Tr\big(g(\hH^{\DD_{W,L}})\big)=\sum_j g(\lambda_j)$,
or use the kernel polynomial methods to calculate the diagonal of the matrix $g\big(\hH^{\DD_{W,L}}\big)$.
The cost of the approximation \eqref{def:dos-reciprocalWL} scales like $\mathcal{O}\big((LW)^{3d}\big)$ if the matrix $\hH^{\DD_{W,L}}$ is diagonalized to obtain all the eigenvalues $\{\lambda_j\}$.
The following theorem justifies the convergence of the planewave approximation \eqref{def:dos-reciprocalWL} and provides the convergence rates with respect to the cutoffs $L$ and $W$.
The proof is given in Appendix \ref{sec:prooftheo:WLconvergence-dos}.

\begin{theorem}
\label{theorem:WLconvergence-dos}
Let $g\in \setg$, then there exist positive constants $C$ and $c$ independent of $L$, $W$ and $g$, such that
\begin{equation}
\label{errordos-reciprocal}
\left\vert 
\pwdosWL{\hH^{\DD_{W,L}}} - \aTr\big(g(H)\big) 
\right\vert \leq
\pmb{\phi}(L) + C \delta^{-2} e^{-c \zeta W} ,
\end{equation}
where $\pmb{\phi}$ is a decreasing function with $\pmb{\phi}(L)\rightarrow 0$ as $L\rightarrow\infty$.
\end{theorem}

\begin{remark}[convergence rate of $\pmb{\phi}(L)$]
\label{remark:oL}
We show in Theorem \ref{theorem:WLconvergence-dos} that as $L\rightarrow\infty$, the approximation \eqref{def:dos-reciprocalWL} can converge to the DoS of the system.
In fact, the term $\pmb{\phi}(L)$ could behave like $\mathcal{O}(L^{-1})$ in practical calculatios of many incommensurate systems (see numerical experiments Section \ref{sec:comput}).
We also discuss in the proof (see Appendix \ref{sec:prooftheo:WLconvergence-dos}) about where the $L^{-1}$ convergence rate comes from and why it can not be justified rigorously.
For discussions on computational cost and numerical experiments in the rest of this paper, we will view $\pmb{\phi}(L)$ as $\mathcal{O}(L^{-1})$.

Then Theorem \ref{theorem:WLconvergence-dos} provides an explicit convergence rate of the planewave approximation \eqref{def:dos-reciprocalWL} of DoS \eqref{tdl}.
We see that the approximate errors decay exponentially fast with respect to $W$ but only decay at most like $\mathcal{O}(L^{-1})$ with respect to $L$.
Therefore, to balance the numerical errors and achieve optimal computational cost, one may need a much larger $L$ than $W$ in practical calculations, see also the shape of $\DD_{W,L}$ in Figure \ref{fig:D_WL}.
By taking $W\ll L$, one can save the degrees of freedom significantly compared with a brute planewave cutoff (e.g. with $|G_1|^2+|G_2|^2\leq L^2$ as used in \cite{zhou2019plane}).
\end{remark}

\begin{remark}[intuition behind the planewave cutoffs]
\label{remark:intution}
Intuitively, increasing $|G_1+G_2|$ (truncated by $W$) corresponds to raising the planewave frequencies, and hence can achieve fast convergence for smooth problems; while increasing $|G_1-G_2|$ (truncated by $L$) corresponds to traversing the reciprocal space, and the convergence of which is determined by the ergodicity and spectral resolution (see the convergence rate in Lemma \ref{lemma:localgeometries}).
These intuitions can be reflected by the proof of this theorem.
\end{remark}

\begin{remark}[comparison between the two algorithms]
\label{remark:comparision}
We make a comparison between the two planewave methods \eqref{discreteddos} and \eqref{def:dos-reciprocalWL} introduced in this section.
As we see from Theorem \ref{theo:convergence_discretizedDos} and \ref{theorem:WLconvergence-dos}, the approximate error of \eqref{discreteddos} decays exponentially fast with respect to the planewave cutoffs while the approximate error of \eqref{def:dos-reciprocalWL} decays at most inversely with respect to the cutoff $L$, so the matrix sizes of those in \eqref{discreteddos} could be significantly smaller than that in \eqref{def:dos-reciprocalWL} to achieve the same accuracy.
Therefore, \eqref{discreteddos} provides a more efficient numerical approximation scheme than \eqref{def:dos-reciprocalWL}, even though in \eqref{discreteddos} one needs to handle many matrices $\hat{H}(\xi)^{\DD_{W,L}}$ with $\xi\in \Kc_{h}^W$.
Moreover, the formula \eqref{discreteddos} can be naturally parallelized as the LDoS for different $\xi$ on the quadrature mesh points can be evaluated in parallel.
The advantage of the formula \eqref{def:dos-reciprocalWL} lies in that it builds a more direct connection to the original continuous problem and the supercell methods \cite{koda2016,komsa2013electronic}.
\end{remark}

\section{Numerical experiments}
\label{sec:comput}
\setcounter{equation}{0}

In this section, we will present some numerical experiments that simulate the DoS of some 1D and 2D incommensurate systems.
All simulations are implemented in open-source {\tt Julia} packages {\tt incommensurate{\_}Pw.jl} \cite{git:Incommensurate_Pw},
and performed on a PC with Intel Core i7-CPU (2.2 GHz) with 32GB RAM.

We test the convergence of numerical approximations of the DoS with some given $g\in\setg$, and show the decay of numerical errors with respect to different parameters.
The results obtained by using sufficiently large discretization parameters (i.e. large $L,W$ and small $h$ in \eqref{discreteddos}) are taken to be the exact solutions.

{\bf Example 1.}
{\rm (1D incommensurate system)}
Considering the following 1D Hamiltonian
\begin{equation}
H= -\frac{1}{2}\frac{\dd^2}{\dd x} + v_1(x) + v_2(x)
\end{equation} 
where the potentials $v_1$ and $v_2$ given by
\begin{equation}
\label{v_numerics}
v_1(x)= \sum_{G_1\in \RL^*_1} e^{-\gamma|G_1|^2} e^{iG_1\cdot x} 
\quad  {\rm and}  \quad 
v_2(x)= \sum_{G_2\in \RL^*_2} e^{-\gamma |G_2|^2} e^{iG_2 \cdot x} ,
\quad x\in\R .
\end{equation}
We take $\gamma=0.01$ and the reciprocal lattices $\RL^*_1=2\pi/L_1 \Z$ and $\RL^*_2=2\pi/L_2 \Z$, with the lattice constants $L_1=\sqrt{5}-1$ and $L_2=2$. 
To compute the DoS that corresponds to the total energy under a Fermi-Dirac distribution, we choose $g$ by
\begin{equation}
\label{gaussian-fermi} 
g_{\beta,\mu}(\lambda) = \frac{\lambda}{1+e^{\beta(\lambda-\mu)}} , 
\end{equation}
where $\mu=10.0$ is a fixed chemical potential and $\beta$ corresponds to the inverse temperature.
We will compare the results with temperature at various values.
Note that the operator $H$ is bounded from below, so it is not necessary to consider the behavior of $g_{\beta,\mu}$ as $x\to -\infty$. 
Therefore, we can view $g$ as a function in $\setg$ with both parameters $\delta$ and $\zeta$ depending on $\beta$.

In the low-temperature regime (corresponding to large $\beta$), $g$ will vanish quickly as $x$ increases. 
Then the convergence of numerical approximations for the DoS will mainly be affected by how far the singularity of $g$ is away from the real axis (that is, by the parameter $\delta$, which will become smaller as $\beta$ increases).
In the high-temperature regime (corresponding to small $\beta$), $g$ will be sufficiently smooth, so the decay of $g$ will play a more important role for the convergence of numerical approximations (that is, by the parameter $\zeta$, which will become smaller as $\beta$ decreases).

We first perform the simulations by the numerical scheme \eqref{discreteddos}. 
We show the decay of numerical errors with respect to the cutoffs $L$ and $W$ in Figure \ref{fig:ex1:WL} (left and middle panels), from which we observe exponential convergence rates with respect to both of the cutoffs.
The various behaviors with different choices of $\beta$ also match our expectation on how the convergence rates depend on $\delta$ and $\zeta$.
We then test the convergence of the planewave approximations by the numerical scheme \eqref{def:dos-reciprocalWL}.
We observe from Figure \ref{fig:ex1:WL} (right panel) that the numerical errors have first order decay with respect to the cutoff $L$, and significantly faster decay with respect to the cutoff $W$.
The results fit our theory in Theorem \ref{theo:convergence_discretizedDos} and Theorem \ref{theorem:WLconvergence-dos} very well.

Finally, we test the convergence with respect to the quadrature mesh size $h$ for the numerical scheme \eqref{discreteddos}.
We show the decay of numerical errors in Figure \ref{fig:ex1:h} for different choices of $W$.
We observe an exponentially fast decay with respect to $h$ when $W$ is large, and a second order decay with respect to $h$ when $W$ is small (where the reference is taken by using the same $W$ and a sufficiently small $h$).
These observations match perfectly with our theoretical predictions in Theorem \ref{theo:convergence_discretizedDos} and Remark \ref{remark:h}.

\begin{figure}[!htb]
\centering
\includegraphics[width=4.8cm]{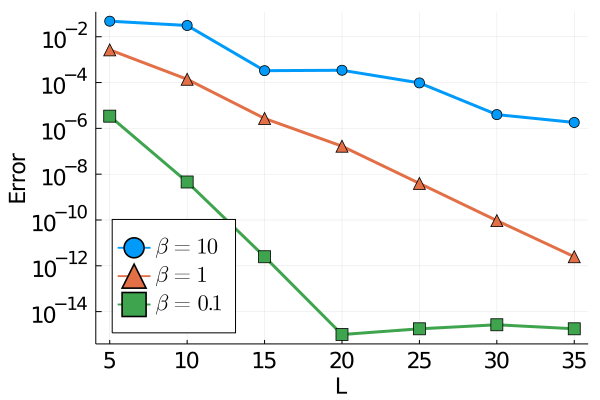}
\hskip 0.15cm
\includegraphics[width=4.8cm]{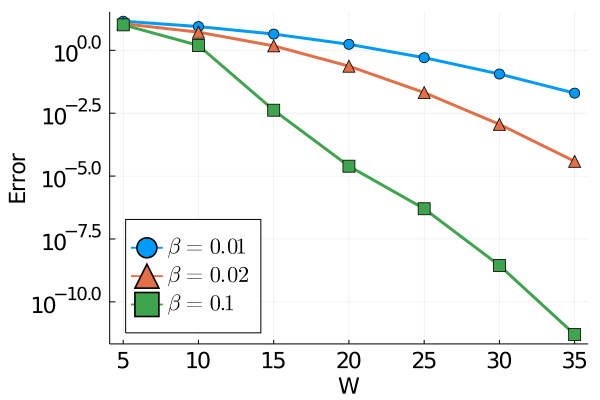}
\hskip 0.15cm
\includegraphics[width=4.9cm]{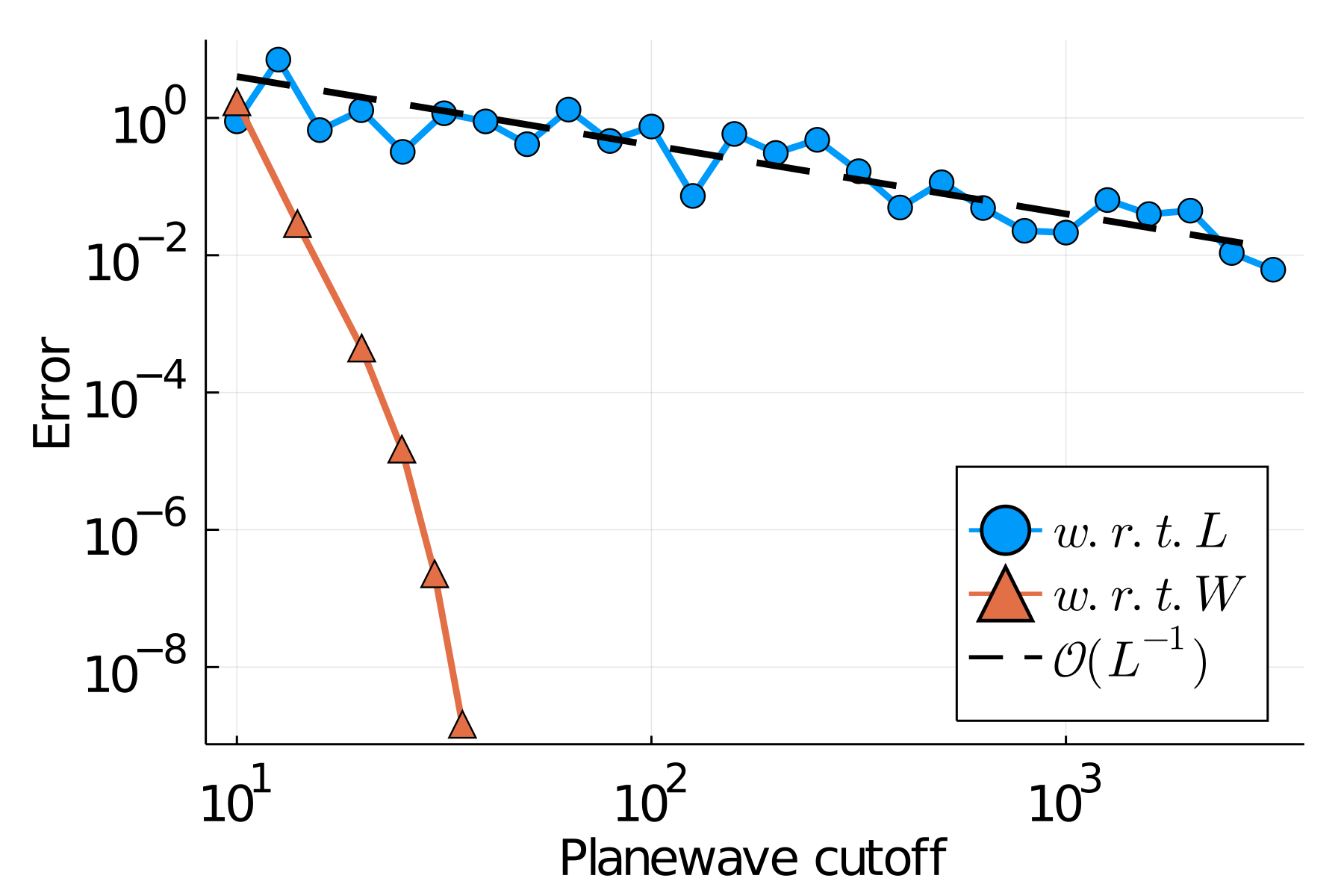}
\caption{(Example 1) Error decay with respect to the planewave cutoffs.
Left: convergence w.r.t. $L$ by \eqref{discreteddos}. 
Middle: convergence w.r.t. $W$ by \eqref{discreteddos}. 
Right: convergence by \eqref{def:dos-reciprocalWL}.}
\label{fig:ex1:WL}
\end{figure}

\begin{figure}[!htb]
\centering
\includegraphics[width=5.0cm]{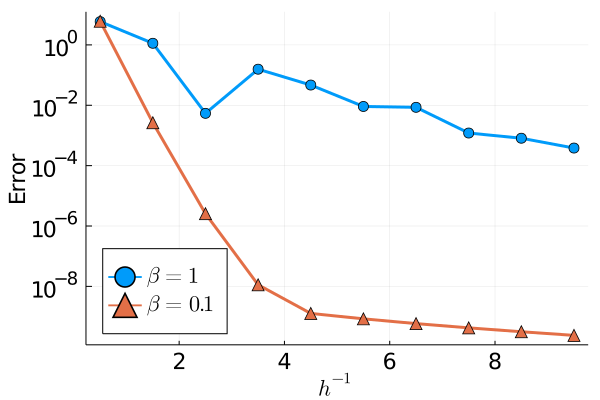}
\hskip 1.0cm
\includegraphics[width=5.0cm]{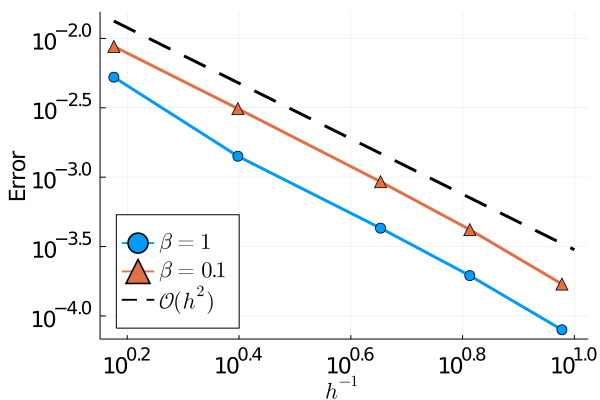}
\caption{(Example 1) Error decay with respect to the quadrature mesh size $h$ by scheme \eqref{discreteddos}. Left: $W=25.0$. Right: $W=5.0$.}
\label{fig:ex1:h}
\end{figure}

{\bf Example 2.} 
{\rm (2D incommensurate system)} 
Consider a two-dimensional incommensurate system obtained by stacking two square lattices together, in which one layer is rotated by an angle $\theta=\pi/10$ with respect to the other. 
More precisely, we take $\mathcal{R}_1= A_1\mathbb{Z}^2$ and $\mathcal{R}_2= A_2\mathbb{Z}^2$ with
\begin{equation*}
	A_1=2\left[ 
	\begin{gathered}
	\begin{matrix}
	1 &0\\0 & 1
	\end{matrix}
	\end{gathered}\right] 
	\qquad{\rm and}\qquad
	A_2=2\left[ 
	\begin{gathered}
	\begin{matrix}
	\cos(\theta) &\cos(\theta+\frac{\pi}{2})  \\\sin(\theta) & \sin(\theta+\frac{\pi}{2}) 
	\end{matrix}
	\end{gathered}\right].	 		 
\end{equation*}
The potentials are given in the same form of \eqref{v_numerics} and the DoS are evaluated by the same test function $g$ as given in \eqref{gaussian-fermi}.

We perform simulations by using the numerical schemes \eqref{discreteddos} and \eqref{def:dos-reciprocalWL}, respectively.
The decay of numerical errors with respect to the planewave cutoffs $L$ and $W$ are shown in Figure \ref{fig:ex2:WL}, which are again consistent with our theory.
Moreover, we present the convergence with respect to the mesh size $h$ for the first scheme \eqref{discreteddos} in Figure \ref{fig:ex2:h}, with a large and a small cutoff of $W$ respectively. 
All these results match perfectly with our theoretical perspective.

\begin{figure}[!htb]
\centering
\includegraphics[width=4.8cm]{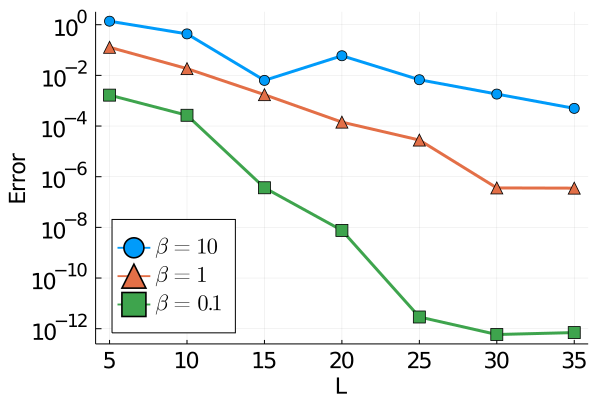}
\hskip 0.15cm
\includegraphics[width=4.8cm]{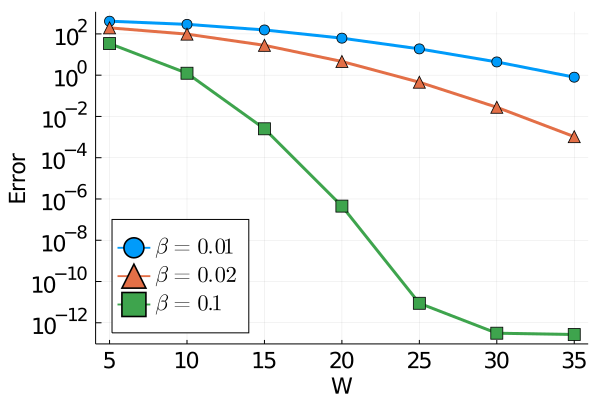}
\hskip 0.15cm
\includegraphics[width=4.9cm]{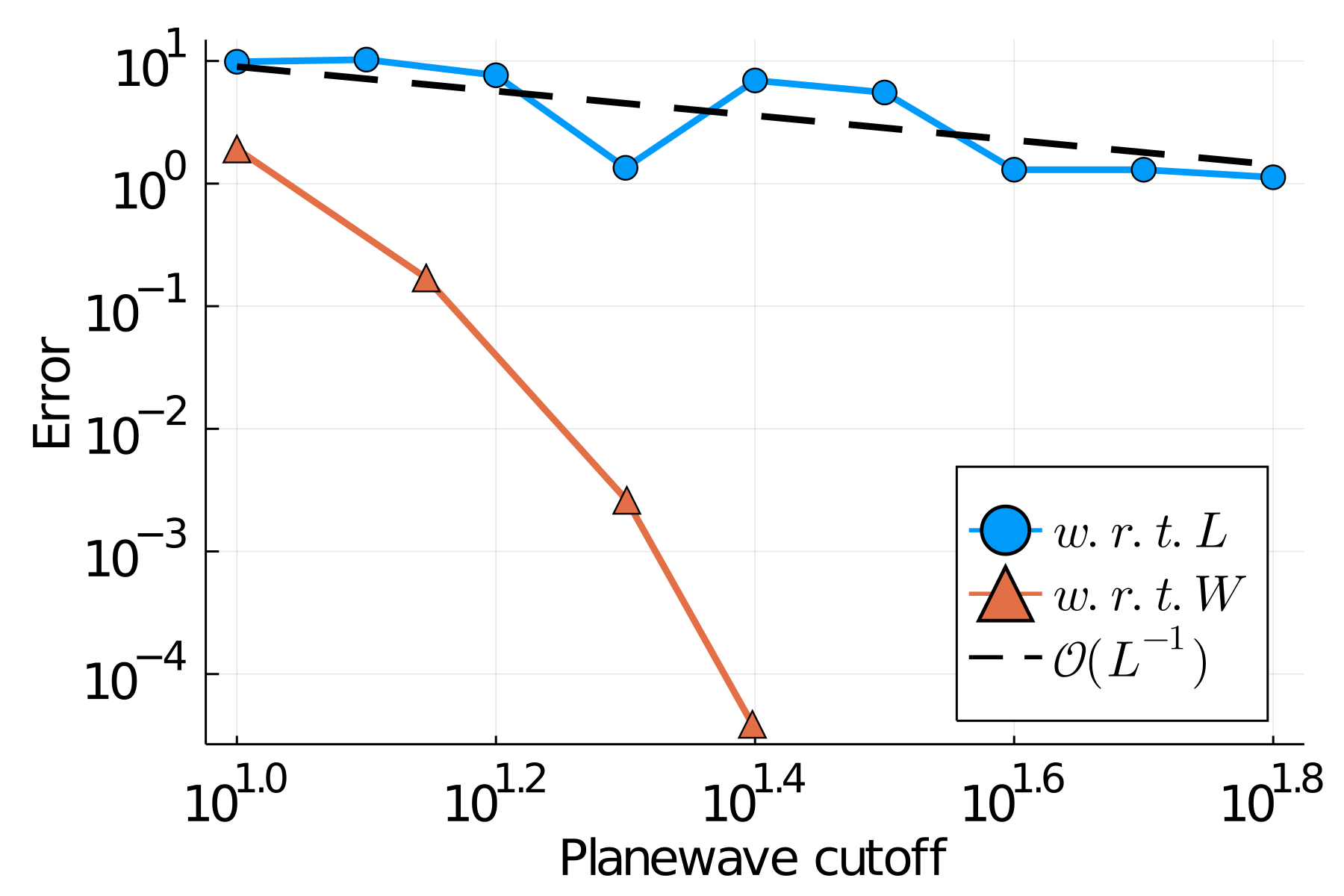}
\caption{(Example 2) Error decay with respect to the planewave cutoffs. 
Left: convergence w.r.t. $L$ by \eqref{discreteddos}. 
Middle: convergence w.r.t. $W$ by \eqref{discreteddos}. 
Right: convergence by \eqref{def:dos-reciprocalWL}.}
\label{fig:ex2:WL}
\end{figure}

\begin{figure}[!htb]
\centering
\includegraphics[width=5.0cm]{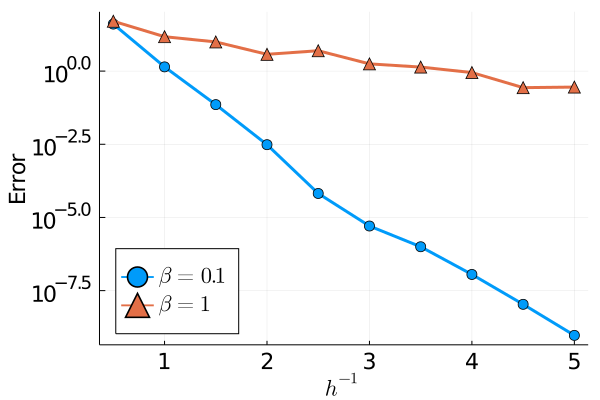}
\hskip 1.0cm
\includegraphics[width=5.0cm]{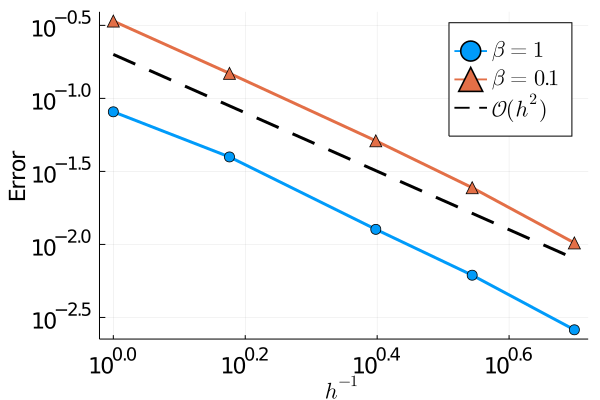}
\caption{(Example 2) Error decay with respect to the quadrature mesh size $h$ by scheme \eqref{discreteddos}. Left: $W=30.0$. Right: $W=5.0$.}
\label{fig:ex2:h}
\end{figure}

\section{Conclusions}
\label{sec:conclusions}

In this paper, we provide a generic framework to study the continuous electronic structure models for incommensurate systems.
We show that the physical observables formulated by the Hamiltonian DoS are well defined and can be simulated efficiently by some planewave approximations.
For future perspective, this work could provide a road map to construct theory and numerical algorithm for more general physical observables, such as conductivity, optical response, and topological Chern number.

In particular, we characterize the DoS of the systems through a planewave description of the operator kernel, and provide explicit DoS representations with respect to the LDoS in real space, reciprocal space, and configuration space. 
Based on the reciprocal space LDoS representation, we develop efficient numerical methods and justify the fast convergence rates of the approximations by rigorous analysis and numerical experiments.
We further provide an alternative planewave approximation scheme, which while less efficient builds a close connection to the continuous Schr\"{o}dinger operator \eqref{H} and earlier works \cite{zhou2019plane}.
In our future work, we will extend to more complex nonlinear models such as density functional theory.

\appendix

\section{Justification and decay property of $g(\hH({\xi}))$}
\label{sec:proof:lemma:gH00}
\setcounter{equation}{0}

We will first justify the elements of $g\big(\hH({\xi})\big)$ as thermodynamic limit in the reciprocal space.
Let $\Omega\subset\RL_1^* \times \RL_2^*$ be a finite set, we denote by $\hH({\xi})^{\Omega}\in\C^{\Omega\times\Omega}$ the ``truncated" Hamiltonian such that its matrix elements are the same as \eqref{Hhat:xi} but the wave vectors $\bG,\bG'$ are restricted to $\Omega$.

\begin{lemma}
\label{lemma:locality:reciprocal}
Let $\xi\in\R^d$, $g\in \setg$ and $\bG,\bG'\in \RL_1^*\times \RL_2^*$.
Let $R>0$ be large enough such that $\max\big\{|\bG|,|\bG'|\big\}<R/2$, and $\Omega_R:= B_R\cap (\RL_1^*\times \RL_2^*)$. 
Then the limit
\begin{eqnarray}
\label{ldos_tdl}
\lodoss{\hH(\xi)}{\bG}{\bG'} := \lim_{R \to \infty}\lodoss{\hH(\xi)^{\Omega_{R}}}{\bG}{\bG'}
\end{eqnarray}
exists.
Moreover, there exist constants $C,c>0$ depending on $\gamma$ (given in \eqref{decayVj}), such that
\begin{eqnarray}
\label{locality:reciprocal}
\left\vert \lodoss{\hH(\xi)^{\Omega_{R}}}{\bG}{\bG'} - \lodoss{\hH(\xi)}{\bG}{\bG'} \right\vert \leq C \delta^{-2} e^{-c\delta R} .
\end{eqnarray}
\end{lemma}	

\begin{proof}
Let $\Omega\subset\RL_1^* \times \RL_2^*$ such that $\Omega\supsetneq\Omega_R$.
We first expand the matrix $\hH({\xi})^{\Omega_{R}}$ to a ``bigger" matrix $\hH(\xi)_{\Omega}^{\Omega_{R}} \in\C^{\Omega\times\Omega}$ by filling zero matrix elements, i.e.
\begin{eqnarray}
\label{hHRmatrixelement}
\big(\hH(\xi)_{\Omega}^{\Omega_{R}}\big)_{\bG',\bG''}=
\left\{
\begin{array}{ll}
\big(\hH(\xi)^{\Omega_{R}}\big)_{\bG',\bG''} \quad &{\rm if \ \bG',\bG''\in \Omega_R},
\\[1ex]
\quad 0 & {\rm otherwise}.
\end{array}
\right.
\end{eqnarray}
Then we have $\df\big(\hH(\xi)_{\Omega}^{\Omega_{R}}\big) = \df\big(\hH(\xi)^{\Omega_{R}}\big)\cup\{0\}$, where $\df(\cdot)$ represents the spectrum set of an operator.

Let $\mathfrak{s}(g)$ represent the non-analytic region of the given function $g\in\setg$.
We can find a contour $\Cc_{R}\subset\C$ such that it encloses all the eigenvalues of $\hH^{\Omega}(\xi)$ and $\hH(\xi)_{\Omega}^{\Omega_{R}}$,
and satisfies
\begin{align}
\label{resolvant-dist}
\min\Bigg\{
{\rm dist}\big(z,\mathfrak{s}(g)\big), ~
{\rm dist}\left(z,\df\big(\hH(\xi)^{\Omega_{R}}\big)\right), ~
{\rm dist}\left(z,\df\big(\hH(\xi)^{\Omega}\big)\right), ~
{\rm dist}\big(z, \{0\}\big) \Bigg\} \geq \frac{\delta}{2}    
\end{align}
for any $z\in\Cc_R$.

Using the decay of $|\hat{V}_{j,G}|$ in \eqref{decayVj}, we have
\begin{equation}
\label{decayhH}
\left \vert \big(\hH(\xi)^{\Omega}\big)_{\bG',\bG''}\right\vert\leq C e^{- \gamma |\bG'-\bG''|}
\quad{\rm and}\quad
\left \vert \big(\hH(\xi)^{\Omega_R}_{\Omega}\big)_{\bG',\bG''}\right\vert\leq C e^{- \gamma |\bG'-\bG''|} .
\end{equation}
Then by using \eqref{resolvant-dist}, \eqref{decayhH} and a Combes–Thomas type estimate \cite{combes1973asymptotic} (see also similar arguments in \cite[Lemma 6]{chen2016qm}), 
we have that there exists a constant $\bar\gamma>0$ depending only on $\gamma$ of \eqref{decayVj}, such that for any $z\in\Cc_R$
\begin{align}
\label{decay_resolvent}
\left\vert \Big(\big(z-\hH(\xi)^{\Omega}\big)^{-1}\Big)_{\bG',\bG''} \right\vert \leq  \frac{1}{\delta} e^{-\bar\gamma \delta|\bG'-\bG''|} ,
\quad
\left\vert \Big(\big(z-\hH(\xi)_{\Omega}^{\Omega_{R}}\big)^{-1}\Big)_{\bG',\bG''} \right\vert \leq \frac{1}{\delta} e^{-\bar\gamma \delta|\bG'-\bG''|} .
\end{align}

By using \eqref{hHRmatrixelement} and \eqref{decayhH}, we have
\begin{align*}
\left|\Big(\hH(\xi)^{\Omega}-\hH(\xi)^{\Omega_R}_{\Omega} \Big)_{\bG',\bG''}\right| \leq \left\{
\begin{array}{ll}
0 & {\rm if}~|\bG'| \leq R ~{\rm and}~|\bG''| \leq R ,
\\[1ex]
C e^{- \gamma |\bG'-\bG''|} & {\rm otherwise} .
\end{array}
\right.
\end{align*}
This together with the decay estimates of resolvents in \eqref{decay_resolvent} implies that, for $\bG,\bG'\in \Omega_{R/2}$
\begin{align}
\label{differenceofresolvent}
& \left| \Big(\big(z-\hH(\xi)^{\Omega}\big)^{-1}\Big)_{\bG,\bG'} -  \Big(\big(z-\hH(\xi)^{\Omega_R}_{\Omega}\big)^{-1}\Big)_{\bG,\bG'} \right| 
\\ \nonumber
\leq & \sum_{\bG''\in\Omega}\sum_{\bG'''\in\Omega} \left| \Big(\big(z-\hH(\xi)^{\Omega_R}_{\Omega}\big)^{-1}\Big)_{\bG,\bG''} 
\Big( \hH(\xi)^{\Omega}-\hH(\xi)^{\Omega_R}_{\Omega} \Big)_{\bG'',\bG'''} 
\Big(\big(z-\hH(\xi)^{\Omega}\big)^{-1}\Big)_{\bG''',\bG'} \right|
\\[1ex] \nonumber
\leq & C \delta^{-2} \left( \sum_{\bG''\in\Omega\backslash\Omega_R}\sum_{\bG'''\in\Omega} + \sum_{\bG''\in\Omega_R}\sum_{\bG'''\in\Omega\backslash\Omega_R} \right)
e^{-\bar\gamma \delta|\bG-\bG''|} e^{- \gamma |\bG''-\bG'''|} e^{-\bar\gamma \delta|\bG'''-\bG'|}
\\[1ex] \nonumber
\leq & C \delta^{-2} e^{-c \delta R} ,
\end{align}
where the constants $C$ and $c$ in the last line depend only on $\gamma$ and $\bar\gamma$.

We can then obtain by using the contour integral representation and the estimate \eqref{differenceofresolvent} that
\begin{align*}
& \left|\lodoss{\hH(\xi)^{\Omega}}{\bG}{\bG'} -\lodoss{\hH(\xi)^{\Omega_{R}(\bG)}}{\bG}{\bG'} \right| 
\\[1ex]
= & \left| \frac{1}{2\pi i}\oint_{\Cc_R} g(z) \bigg( \big(z-\hH^{\Omega}(\xi)\big)^{-1} - \big(z-\hH(\xi)_{\Omega}^{\Omega_R}\big)^{-1} \bigg)_{\bG,\bG'} \dd z \right|
\\[1ex]
\leq & C \delta^{-2} e^{- c\delta R} \oint_{\Cc_R} \big|g(z)\big| \dd z .
\end{align*}
Using the exponential decay of $g$ (indicated in the definition \eqref{deftestfunctiong}), we have that, as $R\rightarrow\infty$, $\oint_{\Cc_R} \big|g(z)\big| \dd z$ is uniformly bounded by some constant depending on $\zeta$.
Note that the above estimate holds for arbitrary $\Omega\supset\Omega_R$.
Therefore, the limit of $\lodoss{\hH(\xi)^{\Omega_{R}(\bG)}}{\bG}{\bG'}$ exists and the estimate \eqref{locality:reciprocal} holds. 
\end{proof}

We can then further show the decay of elements $\lodoss{\hH(\xi)}{\bG}{\bG'}$ as $|\bG-\bG'|$ increases.
This result is used in the Proof of Lemma \ref{lem:boundtracejk}.

\begin{lemma}
\label{lemma:tdl:gH00}
Let $\xi\in\R^d$, $g\in\setg$ and $\bG=(G_1,G_2), ~\bG'=(G'_1,G'_2) \in \RL_1^* \times \RL_2^*$, then there exist positive constants $C$ and $c$ depending on $\zeta$ and $\delta$, such that
\begin{align}
\label{decay-xi-G1-G2}
\left\vert \lodoss{\hH(\xi)}{\bG}{\bG'} \right\vert \leq C e^{- c\big( \min \big\{|\xi+G_1+G_2|,|\xi+G'_1+G'_2| \big\} +  |\bG-\bG'| \big)} .
\end{align}
\end{lemma}

\begin{proof}
Without loss of generality, we assume that $|\xi+G_1+G_2|\leq|\xi+G'_1+G'_2|$.
Then we let $\bar{R}:=\frac{1}{2}|\xi+G_1+G_2|$ and $\Omega_{\bar{R}}(\bG):=B_{\bar{R}}(\bG)\cap\big(\RL_1^*\times\RL_2^*\big)$.
We shall prove \eqref{decay-xi-G1-G2} by considering the two cases $|\bG-\bG'|<\frac{\bar{R}}{2}$ and $|\bG-\bG'|\geq\frac{\bar{R}}{2}$, see the schematic plot in Figure \ref{fig:configs}.

\begin{figure}[!htb]
\centering
\vskip 0.2cm
\includegraphics[width=13.6cm]{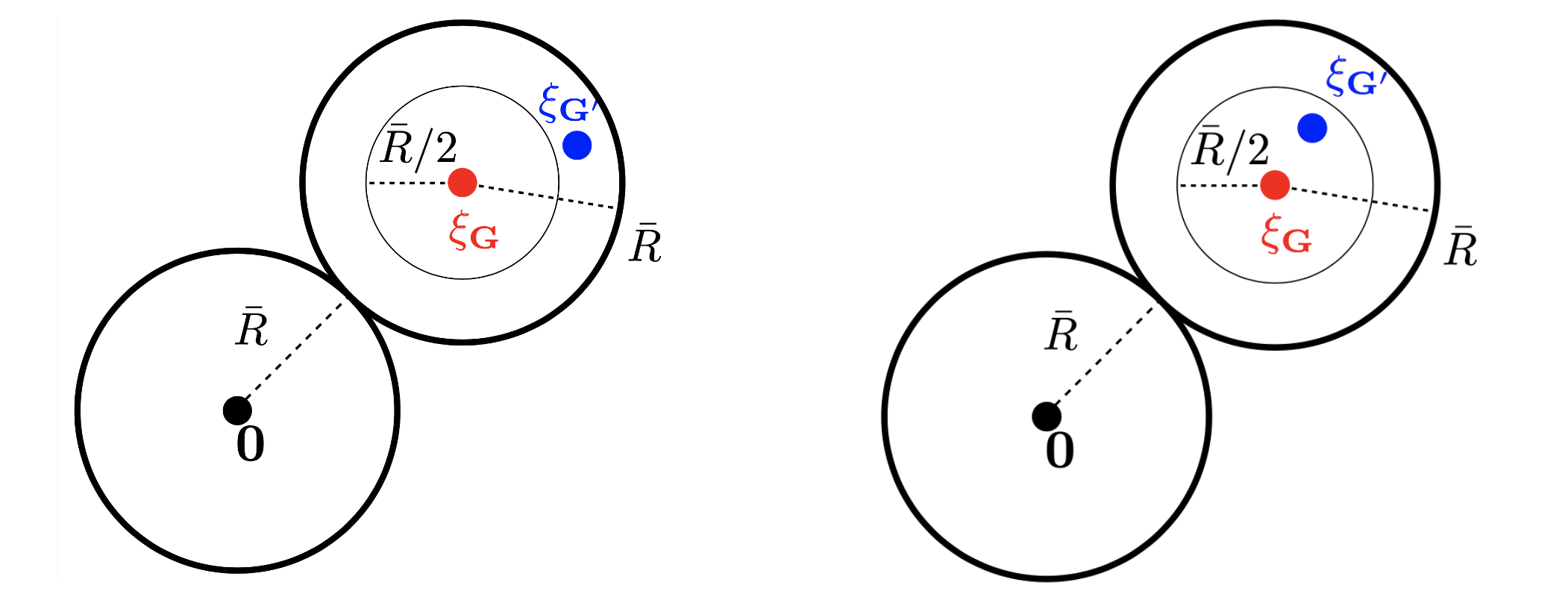}
\caption{For proof of Lemma \ref{lemma:tdl:gH00}, we consider $\xi_{\bf G} = \xi + G_1 + G_2$ and $\xi_{G'} = \xi + G_1' +G_2'$ as displayed. Left: the case of $|\xi_{\bf G} - \xi_{\bf G'}| < \bar R/2$. Right: the case of $|\xi_{\bf G} - \xi_{\bf G'}| \geq \bar R/2$.}
\label{fig:configs}
\end{figure}

We first consider the case when $|\bG-\bG'|<\frac{\bar{R}}{2}$. 
Then we have $|(G'_1+G'_2)-(G_1+G_2)|\leq |\bG-\bG'|<\frac{\bar{R}}{2}$.
By using the decay of $|\hat{V}_{j,G}|$ in \eqref{decayVj}, we see that there exists a constant $a>0$ such that $\sum_{j=1}^2\sum_{G_j\in \RL_j^*}|\hat{V}_{j,G}|<C_V$. 
This together with the definition \eqref{Hhat:xi} and the Ger\v{s}gorin's theorem \cite{horn2012matrix} implies that 
\begin{align*}
\df\big(\hH(\xi)^{\Omega_{\bar{R}}(\bG)}\big) \subset \bigcup_{(G''_1,G''_2)\in \Omega_{\bar{R}}(\bG)} 
\Bigg\{\lambda\in \R: ~
\left\vert \lambda - \frac{1}{2}\left\vert \xi + G''_1+G''_2 \right\vert^2 \right\vert \leq C_V \Bigg\} .
\end{align*}
Therefore, for $g\in\setg$, there exists a contour $\Cc$ that encloses all the eigenvalues of $\hH(\xi)^{\Omega_{\bar{R}}(\bG)}$ and satisfies
\begin{align}
\nonumber
& \min\Big\{ {\rm dist}\big(z,\mathfrak{s}(g)\big) ,~ {\rm dist}\big(z,\df\big(\hH(\xi)^{\Omega_{\bar{R}}(\bG)}\big) \Big\} > \frac{\delta}{2}
\qquad\qquad{\rm and}
\\[1ex]
\label{proof:est:Rez}
& \big| {\rm Re}(z) \big| \geq \min\limits_{(G''_1,G''_2)\in\Omega_{\bar{R}}(\bG)} \Bigg\{\frac{1}{2}|\xi+G''_1+G''_2|^2 \Bigg\} - C_V - 1
\end{align}
for any $z\in\Cc$.
We observe from $\bar{R}=\frac{1}{2}|\xi+G_1+G_2|$ that
\begin{equation}
\min\limits_{(G''_1,G''_2)\in\Omega_{\bar{R}}(\bG)} \frac{1}{2}|\xi + G_1''+G_2''|^2 \geq \frac{1}{2} \bar R^2 = \frac{1}{8} |\xi + G_1 +G_2|^2.
\end{equation}
This together with \eqref{proof:est:Rez} implies
\begin{align}
\label{proof-Rez-estimate}
|{\rm Re}(z)| \geq \frac{1}{8}|\xi+G_1+G_2|^2 - C_V - 1 \geq \frac{1}{8}|\xi+G_1+G_2| - \bar{C}_V
\qquad\forall~z\in\Cc ,
\end{align}
where $\bar{C}_V$ is a constant depending only on $C_V$.
Then by combining the decay property of $g$ (indicated in \eqref{deftestfunctiong}), the estimate of resolvent as that of \eqref{decay_resolvent}, the contour integral representation and \eqref{proof-Rez-estimate}, we have 
\begin{align}
\label{proof-decay-gH-xi}
\left\vert \lodoss{\hH(\xi)^{\Omega_{\bar{R}}(\bG)}}{\bG}{\bG'} \right\vert 
& \leq \frac{1}{2\pi} \oint_{\Cc} |g(z)| \cdot \left\vert \Big(\big(z-\hH(\xi)^{\Omega_{\bar{R}}(\bG)}\big)^{-1}\Big)_{\bG,\bG'} \right\vert \dd z
\\[1ex] \nonumber
& \leq C \frac{1}{2\pi} \oint_{\Cc} e^{-\zeta\big(\frac{1}{8}|\xi+G_1+G_2| - \bar{C}_V\big)} \frac{1}{\delta} e^{-\bar\gamma \delta|\bG-\bG'|} \dd z
\\[1ex] \nonumber
& \leq C \delta^{-1} e^{-c\zeta |\xi+G_1+G_2|}  e^{-\bar\gamma \delta|\bG-\bG'|} .
\end{align}
This together with the estimate \eqref{locality:reciprocal} in Lemma \ref{lemma:locality:reciprocal} implies that
\begin{align}
\label{decay-gH-xi}
\left\vert \lodoss{\hH(\xi)}{\bG}{\bG'} \right\vert 
\leq C\delta^{-1} e^{-c\zeta |\xi+G_1+G_2|} e^{-c\delta|\bG-\bG'|} + C\delta^{-2} e^{-c\delta\bar{R}} ,
\end{align}
which together with the fact $|\bG-\bG'|<\frac{\bar{R}}{2}=\frac{1}{4}|\xi+G_1+G_2|$ leads to
\eqref{decay-xi-G1-G2}. 

In the case when $|(G'_1+G'_2)-(G_1+G_2)| \geq \frac{\bar{R}}{2}$, the proof is actually simpler.
It is only necessary for us to take a sufficiently large $\tilde{R}$ such that $\bG'\in\Omega_{\tilde{R}}(\bG)$.
By using the estimate of \eqref{decay_resolvent} and the contour integral representation, we have that for $g\in\setg$,
\begin{align*}
\left\vert \lodoss{\hH(\xi)^{\Omega_{\tilde{R}}(\bG)}}{\bG}{\bG'} \right\vert 
\leq C\sup_{z\in\Cc} \left\vert \Big(\big(z-\hH(\xi)^{\Omega_{\tilde{R}}(\bG)}\big)^{-1}\Big)_{\bG,\bG'} \right\vert 
\leq C\delta^{-1} e^{-\bar\gamma \delta|\bG-\bG'|} .
\end{align*}
Then by using the estimate \eqref{locality:reciprocal} in Lemma \ref{lemma:locality:reciprocal} and the fact that $|\bG-\bG'|\geq \frac{\bar{R}}{2} = \frac{1}{2}|\xi+G_1+G_2|$, we can obtain the estimate \eqref{decay-xi-G1-G2}.
This completes the proof of Lemma \ref{lemma:tdl:gH00}.
\end{proof}

\vskip 0.2cm

In the rest of this appendix, we shall prove the equation \eqref{connection} in Remark \ref{remark:Hxi}, that is, $\unfold_\xi \big(\widehat{H\psi}\big) = \hH(\xi) \big(\unfold_\xi\hat\psi\big)$ for any $\psi\in \Sc(\R^d)$.

\begin{proof}[Proof of \eqref{connection}]
For ${\bf G} = (G_1,G_2) \in \RL_1^* \times \RL_2^*$, we have from \eqref{V_series} and the definition of $\unfold_\xi$ that
\begin{align*}
& \Big( \unfold_\xi \big(\widehat{H\psi}\big) \Big)(\bG)
= \big(\widehat{H\psi}\big)(\xi+G_1+G_2) 
= \int_{\R^d} \big(H\psi\big)(x) e^{-i(\xi+G_1+G_2)\cdot x} \dd x
\\[1ex]\nonumber
= & \int_{\R^d} \Big(-\frac{1}{2}\Delta\psi(x) + v_1(x)\psi(x) + v_2(x)\psi(x) \Big) e^{-i(\xi+G_1+G_2)} \dd x
\\[1ex]\nonumber
= & \frac{1}{2}\big|\xi+G_1+G_2\big|^2 \hat{\psi}(\xi+G_1+G_2)
+ \sum_{j\in\{1,2\}} \sum_{G'_j\in\RL^*_j} \hat{V}_{j,G'_j} \int_{\R^d} \psi(x)e^{-i(\xi+G_1+G_2-G_j')} \dd x
\\[1ex]\nonumber
= & \frac{1}{2}\big|\xi+G_1+G_2\big|^2 \big(\unfold_{\xi}\hat{\psi}\big)(\bG) + \sum_{G'_1\in\RL^*_1} \hat{V}_{1,G_1-G'_1} \big(\unfold_{\xi}\hat{\psi}\big)(G'_1,G_2) + \sum_{G'_2\in\RL^*_2} \hat{V}_{2,G_2-G'_2} \big(\unfold_{\xi}\hat{\psi}\big)(G_1,G'_2) .
\end{align*}
This together with the definition \eqref{Hhat:xi} of $\hH(\xi)$ implies 
\begin{align*}
\Big( \unfold_\xi \big(\widehat{H\psi}\big) \Big)(\bG)
= \sum_{\bG'\in\RL_1^*\times\RL_2^*} \big[\hH(\xi)\big]_{\bG,\bG'} \hat{\psi}(\xi+G_1'+G_2')
= \Big(\hH(\xi) \big(\unfold_\xi\hat\psi\big)\Big)(\bG)
\quad\forall~ \bG\in\RL_1^* \times \RL_2^* ,     
\end{align*}
and hence leads to \eqref{connection}.
\end{proof}

\section{Proofs}
\label{sec:proof}
\setcounter{equation}{0}

\subsection{Proof of Lemma \ref{lem:boundtracejk}}
\label{sec:proof:lemma:chigH}
To prove Lemma \ref{lem:boundtracejk}, we shall prove that the operator $\chi_j g(H)\chi_k$ is trace class, it has a kernel $K_{jk}(x,y)$ that is continuous and satisfies \eqref{K_jk}, and the trace of $\chi_j g(H)\chi_k$ satisfies the estimate \eqref{lem:trjk}.

We first show that the real-space kernel of $\chi_jg(H)\chi_k$ is \eqref{K_jk}.
To do this, we shall show that 
\begin{align}
\label{proof-kernel}
& \big((\chi_j g(H)\chi_k)\psi\big)(x) = \int_{\R^d} K_{jk}(x,y) \psi(y)\dd y 
\qquad\qquad{\rm with}
\\ \nonumber
& K_{jk}(x,y) = \frac{1}{(2\pi)^d} \chi_j(x)\chi_k(y) \sum_{{\bf G} \in \RL_1^*\times \RL_2^*}e^{-i (G_1+G_2)\cdot y} \int_{\R^d}e^{i\xi\cdot(x-y)} \big[g(\hH(\xi))\big]_{{\bf 0},{\bf G}} \dd\xi .
\end{align}
Using \eqref{connection} and the definition Fourier transform, we have $\FT H \FT^{-1} = \unfold_\xi^{-1} \hH(\xi) \unfold_\xi$ on $\Sc(\R^d)$.
Then we have from the spectrum theory that for $g\in\setg$,
\begin{align*}
\FT g(H) \FT^{-1} = \unfold_\xi^{-1} g\big(\hH(\xi)\big) \unfold_\xi .
\end{align*}
This indicates that for any $\xi\in\R^d$ and $\psi\in \Sc(\R^d)$,
\begin{eqnarray}
\label{unfold}
\Big(\FT\big(g(H)\psi\big)\Big)(\xi) = \sum_{\bG\in\RL_1^* \times \RL_2^*} \big[g(\hat{H}(\xi))\big]_{\pmb{0},\bG}\hat{\psi}(\xi+G_1+G_2) .
\end{eqnarray}
Then we obtain
\begin{align*}
& \big(g(H)(\chi_k\psi)\big)(x)    
= \FT^{-1}\Big(\FT\big(g(H)(\chi_k\psi)\big)\Big)(x)
\\[1ex]
& \qquad = \frac{1}{(2\pi)^d} \int_{\R^d} e^{i\xi\cdot x} \sum_{\bG\in\RL_1^* \times \RL_2^*} \big[g(\hat{H}(\xi))\big]_{\bzero,\bG}\widehat{\chi_k\psi}(\xi+G_1+G_2) \dd\xi
\\[1ex]
& \qquad = \frac{1}{(2\pi)^d} \int_{\R^d} e^{i\xi\cdot x} \sum_{\bG\in\RL_1^* \times \RL_2^*} \big[g(\hat{H}(\xi))\big]_{\bzero,\bG} \int_{\R^d} e^{-i(\xi+G_1+G_2)\cdot y} \chi_k(y)\psi(y) \dd y\dd\xi ,
\end{align*}
which implies the expression of $K_{jk}$ in \eqref{proof-kernel} by multiplying both sides with $\chi_j(x)$.
Note that $\Sc(\R^d)$ is dense in $L^2(\R^d)$, we complete the proof of \eqref{K_jk}.

We shall then show that $\chi_j g(H)\chi_k$ is trace class.
Let $\Lambda := \sqrt{1-\triangle}$ for $\xi\in\R^d$.
We decompose the operator as follows,
\begin{align*}
\chi_j g(H) \chi_k = \big(\chi_j \Lambda^{-1}\big) \big(\Lambda g(H) \Lambda\big) \big(\Lambda^{-1} \chi_k\big).    
\end{align*}
We have from \cite[Theorem 3.8.6]{simon2015operator} that $\chi_j\Lambda^{-1}$ and $\Lambda^{-1}\chi_k$ are Hilbert-Schmidt operators. 
Then to prove that $\chi_j g(H)\chi_k$ is trace class, by using \cite[Theorem 3.8.2]{simon2015operator}, it is only necessary for us to show that $\Lambda g(H) \Lambda$ is a bounded operator. 
To to this, we will show that for any $\psi \in L^2(\R^d)$, it holds that $\Lambda g(H) \Lambda\psi\in L^2(\R^d)$ and $\|\Lambda g(H) \Lambda\psi\|_{L^2}\leq C\|\psi\|_{L^2}$.
Let $\big\langle \xi \big\rangle := \sqrt{1 + |\xi|^2}$ for $\xi\in\R^d$, we observe that $\FT(\Lambda\psi)(\xi) = \big\langle \xi \big\rangle \hat{\psi}(\xi)$ for any $\xi\in\R^d$.
Then by using \eqref{unfold}, we have
\begin{align*}
\FT\big( \Lambda g(H) \Lambda\psi \big)(\xi) 
& = \big\langle \xi \big\rangle \sum_{{\bf G} \in \RL_1^*\times\RL_2^*} \big[g(\hat{H})(\xi)\big]_{\bzero,\bG} \widehat{\Lambda\psi}(\xi + G_1 + G_2)
\\[1ex]
& = \big\langle \xi \big\rangle \sum_{{\bf G} \in \RL_1^*\times\RL_2^*} \big[g(\hat{H})(\xi)\big]_{\bzero,\bG} \big\langle \xi + G_1 + G_2 \big\rangle \hat \psi(\xi + G_1 + G_2) .
\end{align*}
We obtain from the decay estimate \eqref{decay-xi-G1-G2} in Lemma \ref{lemma:tdl:gH00} that
\begin{align*}
\big|\FT(\Lambda g(H) \Lambda \psi)(\xi)\big| \leq \sum_{{\bf G} \in \RL_1^*\times \RL_2^*}  \langle \xi\rangle \langle \xi + G_1+G_2\rangle e^{-c{\rm min}\big\{|\xi|,|\xi + G_1+G_2|\big\}} e^{-c|{\bf G}|} \big|\hat \psi(\xi + G_1+G_2)\big| ,
\end{align*}
which together with the fact ${\rm min} \big\{ |\xi|,|\xi + G_1+G_2| \big\} + |{\bf G}| \geq {\rm max} \big\{ |\xi|,|\xi + G_1+G_2| \big\}$ leads to
\begin{align*}
\big| \FT(\Lambda g(H) \Lambda \psi)(\xi) \big| 
\leq C \sum_{{\bf G} \in \RL_1^*\times \RL_2^*} e^{-c \big({\rm min} \{ |\xi|,|\xi+G_1+G_2| \} + |{\bf G}|\big)} \big|\hat \psi(\xi + G_1+G_2)\big| .
\end{align*}
By squaring both sides and applying the Cauchy Schwarz inequality, we get
\begin{align*}
\Big| \FT\big(\Lambda g(H) \Lambda)\psi\big)(\xi) \Big|^2 
\leq C \sum_{{\bf G} \in \R_1^*\times \R_2^*} e^{-c_2 (|\xi| + |{\bf G}|)} \big|\hat \psi(\xi+G_1+G_2)\big|^2 .
\end{align*}
Integrating both sides of the last estimate with respect to $\xi$, we obtain $\big\|\FT(\Lambda g(H) \Lambda\psi)\big\|_2 \leq C\|\hat\psi\|_2$, and hence $\|\Lambda g(H) \Lambda\psi\|_{L^2}\leq C\|\psi\|_{L^2}$.
This indicates that $\Lambda g(H) \Lambda$ is a bounded operator over $L^2(\R^d)$.

Next we show the continuity of the kernel $K_{jk}$. 
Let $\eps_1,\eps_2\in\R^d$ be small vectors.
We have by using \eqref{K_jk} and the decay estimate \eqref{decay-xi-G1-G2} in Lemma \ref{lemma:tdl:gH00} that
\begin{align*}
& \big| K_{jk}(x+\eps_1,y+\eps_2) - K_{jk}(x,y) \big| 
\\[1ex]
= & \frac{1}{(2\pi)^d} \Bigg|
\sum_{{\bf G} \in \RL_1^*\times \RL_2^*} e^{-i(G_1+G_2)\cdot y} 
\int_{\R^d} \bigg( \chi_j(x+\eps_1)\chi_k(y+\eps_2) e^{-i(G_1+G_2)\cdot\varepsilon_2} e^{i\xi \cdot (\eps_1-\eps_2)} 
\\[1ex]
& \qquad\qquad \qquad\qquad 
- \chi_j(x)\chi_k(y) \bigg) 
~ \big[g(\hH(\xi))\big]_{{\bf 0},{\bf G}} \dd\xi \Bigg|
\\[1ex]
\leq & C \sum_{{\bf G} \in \RL_1^*\times \RL_2^*} e^{-c|\bG|}
\int_{\R^d} \bigg( \left| \chi_j(x+\eps_1)\chi_k(y+\eps_2) \Big(e^{-i(G_1+G_2)\cdot\varepsilon_2} e^{i\xi \cdot (\eps_1-\eps_2)} - 1\Big) \right| 
\\[1ex]
& \qquad\qquad \qquad\qquad 
+ \Big| \chi_j(x+\eps_1)\chi_k(y+\eps_2) - \chi_j(x)\chi_k(y) \Big| \bigg) 
~ e^{-c\min\big\{|\xi|,|\xi+G_1+G_2|\big\}} \dd\xi .
\end{align*}
We observe that the right-hand side of the above estimate goes to zero as $\eps_1,\eps_2 \rightarrow \bzero$, which hence verifies the continuity of $K_{jk}$.

Finally, we shall prove the estimate \eqref{lem:trjk}.
Since $\chi_j$ has compact support, we have that for any $j,k\in\Z^d$, the operator $\chi_j g(H)\chi_k$ can be considered as an operator on a compact domain, and hence effectively has an associated probability measure over the domain.
Note that we have also shown that $\chi_j g(H)\chi_k$ is trace class and its kernel $K_{jk}$ is continuous.
Therefore, by using the result in \cite[Proposition 3.11.2]{simon2015operator}, we have that the trace of $\chi_j g(H)\chi_k$ can be given by its kernel
\begin{align}
\label{proof-tc-c}
{\rm Tr}\big(\chi_j g(H)\chi_k\big)
& = \int_{\R^d} K_{jk}(x,x) \dd x
\\ \nonumber
& = \frac{1}{(2\pi)^d} \sum_{{\bf G} \in \RL_1^*\times \RL_2^*} \int_{\R^d} \big[g(\hH(\xi))\big]_{\bzero,\bG} \dd\xi \int_{\R^d}\chi_j(x)\chi_k(x)e^{-i (G_1+G_2)\cdot x} \dd x .
\end{align}
Since $\chi_j(x)\chi_k(x)=0$ when $|j-k|>1$, we have from \eqref{proof-tc-c} that ${\rm Tr}\big(\chi_j g(H)\chi_k\big)
=0$ when $|j-k|>1$.
When $|j-k|\leq 1$, we can obtain by using the decay estimate \eqref{decay-xi-G1-G2} in Lemma \ref{lemma:tdl:gH00} that
\begin{align*}
{\rm Tr}\big(\chi_j g(H)\chi_k\big) 
\leq \sum_{{\bf G} \in \RL_1^*\times \RL_2^*} \int_{\R^d} \big[g(\hH(\xi))\big]_{{\bf 0},{\bf G}} \dd\xi 
\leq C \sum_{{\bf G} \in \RL_1^*\times \RL_2^*} e^{-c|\bG|} \int_{\R^d} e^{-c |\xi+G_1+G_2| } \dd\xi ,
\end{align*}
which can be bounded by a constant $C$.
This completes the proof of Lemma \ref{lem:boundtracejk}.
$\hfill\square$

\subsection{Proof of Theorem \ref{theo:tdllimit}}
\label{sec:proof:theo:tdllimit}

We shall prove that the limit of \eqref{def:dosR} exists as $R\rightarrow\infty$ and
\begin{align}
\label{dos:gH00}
\aTr\big(g(H)\big) := \lim_{R\rightarrow\infty}\aTr_R\big(g(H)\big)
= \int_{\R^d} [g(\hH(\xi))]_{{\bf 0,0}} \dd\xi.
\end{align}

Let $\chi_R(x):=\sum_{j\in\Z^d\cap B_R} \chi_j(x)$.
Note that Lemma \ref{lem:boundtracejk} implies that the trace of $\chi_jg(H)\chi_k$ is well defined for any $j,k \in \Z^d$.
This together with the definition of $\underline{\rm Tr}_R(g(H))$ in \eqref{def:dosR} 
and the expression of kernel $K_{jk}$ in \eqref{K_jk} leads to
\begin{align*}
& \hskip -0.5cm
\underline{\rm Tr}_R(g(H)) 
= \frac{1}{|B_R|}\sum_{j,k\in\Z^d\cap B_R}{\rm Tr}\big(\chi_jg(H)\chi_k\big)
\\[1ex]
& = \frac{1}{|B_R|}\sum_{j,k\in\Z^d\cap B_R} \int_{\R^d} K_{jk}(x,x)\dd x
\\[1ex]
&= \frac{1}{(2\pi)^d}\frac{1}{|B_R|} \sum_{j,k\in\Z^d\cap B_R}
\sum_{{\bf G} \in \RL_1^*\times \RL_2^*} \int_{\R^d} \chi_j(x)\chi_k(x) e^{-i (G_1+G_2)\cdot x} \dd x
\int_{\R^d} \big[g(\hH(\xi))\big]_{{\bf 0},{\bf G}} \dd\xi
\\[1ex]
&= \frac{1}{(2\pi)^d}\frac{1}{|B_R|} \Bigg(\int_{\R^d} \chi^2_R(x) \dd x \Bigg)
\Bigg(\int_{\R^d} \big[g(\hH(\xi))\big]_{{\bf 0},{\bf 0}} \dd\xi \Bigg)
\\[1ex]
& \quad + \frac{1}{(2\pi)^d} \frac{1}{|B_R|} 
\sum_{{\bf G} \in (\RL_1^*\times \RL_2^*)\backslash \{\bzero\}} 
\Bigg(\int_{\R^d} \chi^2_R(x) e^{-i (G_1+G_2)\cdot x} \dd x \Bigg)
\Bigg(\int_{\R^d} \big[g(\hH(\xi))\big]_{\bzero,\bG} \dd\xi \Bigg)
\\[1ex]
& =: T_1 + T_2 .
\end{align*}

For the first term $T_1$, by using the fact $\chi_R(x)=1$ when $|x|\leq R$ and $\chi_R(x)=0$ when $|x|\geq R+1$, we have $\displaystyle  \lim_{R\rightarrow\infty}\frac{1}{|B_R|}\int_{\R^d}\chi_R^2(x)\dd x = 1$.
Hence the limit of $T_1$ exists as $R\rightarrow\infty$,
\begin{align*}
\lim_{R\rightarrow\infty} T_1 = \int_{\R^d} [g(\hH(\xi))]_{{\bf 0,0}} \dd\xi .
\end{align*}

For the second term $T_2$, we obtain from the decay estimate \eqref{decay-xi-G1-G2} that 
$\int_{\R^d} \big[g(\hH(\xi))\big]_{\bzero,\bG} \dd\xi \leq Ce^{-c|\bG|}$,
which together with the fact $\big|e^{-i (G_1+G_2)\cdot x}\big|\leq 1$ implies that the series
\begin{align*}
\sum_{{\bf G} \in (\RL_1^*\times \RL_2^*)\backslash \{\bzero\}} 
\Bigg(\frac{1}{|B_R|} \int_{\R^d} \chi^2_R(x) e^{-i (G_1+G_2)\cdot x} \dd x \Bigg)
\Bigg(\int_{\R^d} \big[g(\hH(\xi))\big]_{\bzero,\bG} \dd\xi \Bigg)  
\end{align*}
is uniformly convergent.
Moreover, we have that for any $\zeta\in\R^d,~\zeta\neq 0$, it holds that
\begin{align*}
\lim_{R\rightarrow\infty} \frac{1}{|B_R|} \int_{\R^d}\chi_R^2(x)e^{i\zeta\cdot x}\dd x = 0 .
\end{align*}
Note that $\RL_1$ and $\RL_2$ are incommensurate indicates that $(G_1+G_2)\neq 0$ if $\bG\in (\RL_1^*\times \RL_2^*)\backslash \{\bzero\}$.
Therefore, we can derive that the limit of $T_2$ exists as $R\rightarrow\infty$, 
\begin{align*}
\lim_{R\rightarrow\infty} T_2 = 
\sum_{{\bf G} \in (\RL_1^*\times \RL_2^*)\backslash \{\bzero\}} \Bigg(\int_{\R^d} \big[g(\hH(\xi))\big]_{\bzero,\bG} \dd\xi \Bigg)
\lim_{R\rightarrow \infty} \Bigg(\frac{1}{|B_R|} \int_{\R^d} \chi^2_R(x) e^{-i (G_1+G_2)\cdot x} \dd x \Bigg) = 0 .
\end{align*}
This complete the proof of \eqref{dos:gH00}.
\hfill
$\square$

\subsection{Proof of Theorem \ref{theo:convergence_discretizedDos}}
\label{sec:prooftheo:convergence_discretizedDos}

We recall the statement of Theorem \ref{theo:convergence_discretizedDos}, which gives the error estimate of the numerical approximation \eqref{discreteddos}.
Let $g\in\setg$, then there exist positive constants $C$ and $c$ that do not depend on $L$, $W$, $h$ and $g$, such that
\begin{equation*}
\left\vert 
\DosK{W,L}{\DD_{W,L}} -  \aTr\big(g(H)\big) 
\right\vert 
\leq C
\big( \delta^{-2} e^{- c\delta L} + \delta^{-2} e^{- c \zeta W} + e^{- c\delta/h} \big).
\end{equation*}
Before the proof, we first give a lemma that provides a sharper estimate than Lemma \ref{lemma:locality:reciprocal}, but only for the diagonal elements of $g\big(\hH(\xi)\big)$. 
To do this, we introduce a more specific truncation (than $\Omega_R(\bG)$ used in Appendix \ref{sec:proof:lemma:gH00}) of the wave vectors around $\bG$.
Let
\begin{align}
\label{DWL}
\DD_{W,L}(\bG):= \Big\{ \big(G_1^{''},G_{2}^{''}\big) \in\RL_1^* \times \RL_2^* :~ 
& \big|(G_1^{''}-G_1)+(G_2^{''}-G_2)\big|\leq W, 
\\\nonumber
& \big|(G_1^{''}-G_1)-(G_2^{''}-G_2)\big|\leq L \Big\} .
\end{align}

\begin{lemma}
\label{lemma:locality:reciprocal:GG}
Let $\xi\in\R^d$, $g\in \setg$, $\bG=(G_1,G_2)\in \RL_1^*\times \RL_2^*$.
Then there exist constants $C$ and $c$ independent of $W$ and $L$, such that
\begin{eqnarray}
\label{Wlocality:reciprocal}
\left\vert \lodos{\hH(\xi)^{\DD_{W,L}(\bG)}}{\bG} - \lodos{\hH(\xi)}{\bG} \right\vert \leq C \delta^{-2} \big( e^{-c\delta L} + e^{-c \zeta W} \big).
\end{eqnarray}
\end{lemma}	

\begin{proof}
Note that by using the same argument as that for the estimate \eqref{locality:reciprocal} in Lemma \ref{lemma:locality:reciprocal}, we can immediately imply that
\begin{eqnarray}
\label{decay-WL-delta}
\left\vert \lodos{\hH(\xi)^{\DD_{W,L}(\bG)}}{\bG} - \lodos{\hH(\xi)}{\bG}  \right\vert \leq C \delta^{-2} \big( e^{-c\delta L} + e^{-c \delta W} \big) .
\end{eqnarray}
Therefore, it is only necessary for us to modify the convergence rate with respect to $W$, that is,
\begin{eqnarray}
\label{proof-4-1-W}
\left\vert \lodos{\hH(\xi)^{\DD_{W',L}(\bG)}}{\bG} - \lodos{\hH(\xi)^{\DD_{W,L}(\bG)}}{\bG} \right\vert \leq C \delta^{-2} e^{-c\zeta W} 
\qquad\forall~W'>W ,
\end{eqnarray}
where $L$ is fixed and the exponent does not depend on $\delta$ as that in \eqref{locality:reciprocal} or \eqref{decay-WL-delta}.
Without loss of generality, we only consider the case $\xi=0$ in the proof, and denote by $\hH:=\hH(\bzero)$ in the following.

By using the decay of $|\hat{V}_{j,G}|$ in \eqref{decayVj}, we see that there exists a constant $C_V>0$ such that $\sum_{j=1}^2\sum_{G\in \RL_j^*}|\hat{V}_{j,G}|<C_V$. 
This together with the matrix element definition \eqref{Hpw} and the Ger\v{s}gorin's theorem \cite{horn2012matrix} 
implies 
\begin{align*}
\df\big(\hH^{\DD_{W',L}(\bG)}\big) \subset \bigcup_{(G'_1,G'_2)\in \DD_{W',L}(\bG)} 
\left\{ \lambda\in \R: ~ \Big\vert \lambda - \frac{1}{2}\big\vert G'_1+G'_2 \big\vert^2 \Big\vert \leq C_V \right\} .
\end{align*}
Then there exists a contour $\Cc$ that encloses all the eigenvalues of $\hH^{\DD_{W',L}(\bG)}$ and satisfies
\begin{align}
\label{resolvant-dist-decay-W}
\min\Bigg\{ {\rm dist}\big(z,\mathfrak{s}(g)\big), ~
{\rm dist}\left(z,\df\big(\hH^{\DD_{W',L}(\bG)}\big)\right) \Bigg\} \geq \delta
\qquad\forall~z\in\Cc .
\end{align}
We can then apply the ``resolvent decomposition" approach to prove the convergence rate, which was developed and used in \cite{massatt2021electronic,massatt2018incommensurate}.
For the sake of completeness, we will present the details in the rest of this proof.

We first introduce some notations for the ``resolvent decomposition".
Let $r>0$ be a given constant independent of $W$.
Let $n$ and $N$ be the largest integers no greater than $W^2/r$ and $W'^2/r$ respectively.
Let $z\in\Cc$ satisfy 
$\big|{\rm Re}(z)-\frac{1}{2}|G_1+G_2|^2\big|<
\frac{W^2}{4}$.
Define
\begin{align*}
\St_0(z) := \Big\{\big(G_1^{'},G_2^{'}\big)\in\DD_{W',L}(\bG) :~ \Big|{\rm Re} (z)- \frac{1}{2}| G_1^{'}+G_2^{'}|^2\Big|\leq r\Big\} .
\end{align*}
Let $\Kc:=\{1,\dots,N\}$, and define the regions
\begin{equation}
\label{def:Wring}
\St_k(z) := \Big\{\big(G_1^{'},G_2^{'}\big)\in\DD_{W',L}(\bG)\backslash\St_0(z) : ~ k r < \Big|{\rm Re} (z)- \frac{1}{2}| G_1^{'}+G_2^{'}|^2\Big|\leq (k+1) r\Big\}
\end{equation}
for $k\in\Kc$.
We then define a sequence of the corresponding projections $P_k:\R^\DD_{W',L}\rightarrow \R^\DD_{W',L}$ by
\begin{eqnarray*}
\big(P_k \psi\big)_{{\bf G}'} = \left\{
\begin{array}{ll}
\psi_{{\bf G}'} \quad & {\rm if}~{\bf G}'
\in \St_k(z), 
\\[1ex]
0 & {\rm otherwise}.
\end{array}
\right.
\end{eqnarray*}
We can further define the projected Hamiltonian by $\hH_{k\ell}:=P_k\hH P_\ell$ for $k,\ell\in\Kc$.
Finally, we shall denote by $\hH_N:=\hH^{\DD_{W',L}(\bG)}$ and $\hH_n:=\hH^{\DD_{W,L}(\bG)}_{\DD_{W',L}(\bG)}$ (the expanded matrix of $\hH^{\DD_{W,L}(\bG)}$ as defined in \eqref{hHRmatrixelement}).

Let $\alpha\in(0,1)$ be a fixed constant and $\beta:=\sup_k\sum_{\ell\in\Kc\backslash\{k\}}\|\hH_{k\ell}\|_2$.
Due to the decay of $|\hat{V}_{j,G}|$ in \eqref{decayVj}, we have $\|\hH_{k\ell}\|_2\leq Ce^{-c_{\gamma}|k-\ell|}$, where the constant $c_{\gamma}$ depends on $\gamma$ given in \eqref{decayVj}.
%
%
We then determine the constant $r$ such that the following condition holds
\begin{equation}
\label{subresolventbound}
\|(zI_{kk}-\hH_{kk})^{-1}\|_2 \leq \alpha \beta ^{-1} 
\qquad \forall~k\in\Kc .
\end{equation}
For $k,\ell\in\Kc$, we define the projected resolvent by
\begin{equation*}
\Gl_{k\ell}:=P_k\big(zI-\hH_{N}\big)^{-1}P_{\ell}.
\end{equation*} 
We have from ${\rm dist}\big(z,\df\big(\hH_N\big)\big) \geq \delta$ that $\|\Gl_{k\ell}\|_2\leq \delta^{-1}$ for any $k,\ell\in\Kc$.
By applying Schur's complement expansions, we have
\begin{equation}
\label{shurcomplement}
\Gl_{k\ell}=-\big(zI_{kk}-\hH_{kk}\big)^{-1}\hH_{k0} \Gl_{0\ell} - \sum_{s\in\Kc\backslash \{k\}}\big(zI_{kk}-\hH_{kk}\big)^{-1}\hH_{ks}\Gl_{s\ell} .
\end{equation}
Combing the facts $\|\hH_{k\ell}\|_2\leq Ce^{-c_{\gamma}|k-\ell|}$, $\|\Gl_{k\ell}\|_2\leq \delta^{-1}$ with \eqref{subresolventbound} and \eqref{shurcomplement}, we have 
\begin{equation}
\label{shurnorm}
\|\Gl_{k\ell}\|_2 \leq C \alpha\beta^{-1} \delta^{-1} e^{-c_{\gamma}|k|} + \sum_{s\in\Kc\backslash \{k\}}\alpha
\beta^{-1}\|\hH_{ks}\|_2\|\Gl_{s\ell}\|_2. 
\end{equation}

For $\ell<k$, denote $X:=(X_1,\dots,X_N)^{\rm T}\in \R^{N}$ with $X_k: =\|\Gl_{k\ell}\|_2$, and $\Phi:=(\Phi_1,\dots,\Phi_N)^{\rm T}\in \R^{N}$ with $\Phi_k:=e^{-c_{\gamma}|k|}$.
We have $\|X\|_2 \leq \delta^{-1}$.
Let $M\in \R^{N\times N}$ be a matrix with
\begin{equation*}
M_{ks}:=\alpha\beta^{-1}
\|\hH_{ks}\|_2(1-\delta_{ks}).
\end{equation*}
We have $M_{ks}\leq C \alpha e^{-c_{\gamma}|k-s|}(1-\delta_{ks})$ and $\|M\|_2\leq\alpha$.
Then we can rewrite \eqref{shurnorm} as 
\begin{equation*}
X_k \leq C \delta^{-1} \Phi_k + (MX)_k,
\end{equation*}
which indicates $X \leq C \delta^{-1}\Phi + MX$.
Applying this inequality to itself $m$ times, we have
\begin{equation}
\label{proof-decay-W-a}
X \leq C \delta^{-1} \sum_{p=0}^{m-1} M^p\Phi + M^m X.
\end{equation}
To estimate $M^p\Phi$, we choose a contour $\Cc_M$ such that 
$\frac{\alpha_0}{2} \leq {\rm dist}\big(z',[-\|M\|_2, \|M\|_2]\big) \leq \alpha_0$ with some constant $\alpha_0<1$.
By using the exponential decay $M_{ks}\leq C_{\alpha}e^{-c_{\gamma}|k-s|}(1-\delta_{ks})$ and the Combes–Thomas type estimate \cite{combes1973asymptotic,chen2016qm}, we have that there exist a constant $\tilde{c}_{\gamma}$ such that
\begin{equation*}
\big\|\big[(z'-M)^{-1}\big]_{ks}\big\|_2\leq C e ^{-\tilde{c}_{\gamma} |k-s|}
\qquad\forall~z'\in\Cc_M .
\end{equation*}
Then we have
\begin{equation*}
\big|(M^p)_{ks}\big|\leq \bigg| \frac{1}{2\pi i}\oint_{\Cc_M}(z')^p \big[(z'-M)^{-1}\big]_{ks}\dd z \bigg| \leq C \alpha_0^p e^{-\tilde{c}_{\gamma}|k-s|},
\end{equation*}
and hence 
\begin{equation}
\label{proof-decay-W-b}
\left|\sum_{s\in\Kc} (M^p)_{ks} e^{-c_{\gamma} s} \right| \leq C \alpha_0^p\sum_{s\in\Kc} e^{-\tilde{c}_{\gamma}|k-s|} e^{-c_{\gamma} s} 
\leq C \alpha_0^p e^{-\bar{c}_{\gamma} k},
\end{equation}
where $\bar{c}_{\gamma}=\min\{\tilde{ c}_\gamma, c_\gamma\}$.
We then deduce from \eqref{proof-decay-W-a}, \eqref{proof-decay-W-b} and $\|M\|_2\leq \alpha$ that 
\begin{equation}
\label{resolventk}
X_k \leq C \delta^{-1} \alpha_0^m e^{-\bar{c}_{\gamma} k} + \alpha^m \delta^{-1}. 
\end{equation} 
By choosing $m\sim k$ with appropriate balancing constant, we have that there exists a constant $c_0>0$ depending on $\alpha,\beta$ and $\gamma$, such that
\begin{equation}
\label{proof-decay-W-c}
X_k \leq C \delta^{-1} e^{-c_0 k} .
\end{equation}  

Remember that we chose a contour $\Cc$ satisfying \eqref{resolvant-dist-decay-W}, which encloses the spectrum of $\hH_N$. 
We have from \eqref{def:Wring} that there exists $n_0<n$, such that $\bG\in \St_{n_0}(z)$.
Let $\Cc_1=\{z\in \Cc:\big|{\rm Re} (z)-\frac{1}{2}|G_1+G_2|^2\big|<
\frac{W^2}{4}\}$ and $\Cc_2=\Cc\backslash \Cc_1$.
Then we estimate that
\begin{align}
\label{diffldosW}
& P_{n_0}\big(g({\hH_N}) - g({\hH_n})\big) P_{n_0}  
\\[1ex] \nonumber
= & ~ \frac{1}{2\pi i }\oint_{\Cc_1} g(z) 
\Big(P_{n_0}\Big( z-\hH_N \Big)^{-1} \sum_{0<s\leq N} P_s P_s
\Big(\hH_N-\hH_n \Big)\sum_{0<t\leq N} P_t P_t
\Big(z-\hH_n \Big)^{-1} P_{n_0} \Big)\dd z 
\\[1ex] \nonumber
& ~ + \frac{1}{2\pi i }P_{n_0}\bigg(\oint_{\Cc_2} g(z)
\Big(\big(z-\hH_N \big)^{-1} - \big(z-\hH_n \big)^{-1}\Big)\dd z \bigg)P_{n_0}. 
\end{align}
Combining \eqref{proof-decay-W-c}, \eqref{diffldosW}, the decay of $\|\hH_{k\ell}\|_2\leq Ce^{-c_{\gamma}|k-\ell|}$, and the decay of $g$ (indicated in the definition \eqref{deftestfunctiong}), we have
\begin{align*}
& \big\|P_{n_0}\big(g({\hH_N}) - g({\hH_n})\big) P_{n_0}\big\|_2 
\\[1ex]
\leq & \frac{1}{2\pi }\oint_{\Cc_1} |g(z)| \sum_{0\leq s\leq n<t\leq N} C_{\alpha}\delta^{-2} e^{-c_0(s+t)} e^{-c_\gamma|s-t|} \dd z + \delta^{-2} e^{-\zeta W^2}
\\[1ex]
\leq & C\delta^{-2} e^{-c n} + C\delta^{-2} e^{-\zeta W^2} 
\end{align*}
with some constant $c>0$.
Note that this estimate is uniform with respect to $N$. Since $W \sim n^{1/2}$, we conclude that there is a $c>0$ such that
\begin{equation*}
\|P_{n_0}\big(g({\hH_N}) - g({\hH_n})\big) P_{n_0}\|_2 \leq C \delta^{-2} e^{-c \zeta n} .
\end{equation*}
This implies \eqref{proof-4-1-W} and hence completes the proof.
\end{proof}

With the more explicit estimate \eqref{Wlocality:reciprocal} in Lemma \ref{lemma:locality:reciprocal:GG} (than that in Lemma \ref{lemma:locality:reciprocal}), we can immediately modify the estimate \eqref{decay-gH-xi} (by replacing $\Omega_{\bar{R}}(\bG)$ with $D_{\bar{R},\infty}(\bG)$ in the proof)
\begin{align*}
\left\vert \lodoss{\hH(\xi)}{\bG}{\bG'} \right\vert 
\leq C\delta^{-1} e^{-c\zeta |\xi+G_1+G_2|} e^{-c\delta|\bG-\bG'|} + C\delta^{-2} e^{-c\zeta |\xi+G_1+G_2|} .
\end{align*}
In particular, by taking $\bG=\bG'=\bzero$, it holds that
\begin{align}
\label{decay-gH-xi-ext}
\left\vert \lodoss{\hH(\xi)}{\bzero}{\bzero} \right\vert 
\leq C\delta^{-2} e^{-c\zeta |\xi|} .
\end{align}
Then we are ready to prove Theorem \ref{theo:convergence_discretizedDos}.

\begin{proof}[Proof of Theorem \ref{theo:convergence_discretizedDos}]
Let $\Kc_h := h\Z^d$ be a uniform mesh over $\R^d$.
We see from \eqref{domainofsampling} that the quadrature mesh can be written as $\Kc_h^W=\Kc_h\cap[-W,W]^d$.

Note that Theorem \ref{theo:tdllimit} implies $\aTr\big(g(H)\big)
= \int_{\R^d} [g(\hH(\xi))]_{{\bf 0,0}} \dd\xi$. 
Then by using the definition \eqref{discreteddos}, we derive that 
\begin{align}
\label{proof-4-3-a}
& \left\vert \DosK{W,L}{\DD_{W,L}} - \aTr\big(g(H)\big) \right\vert 
\\ \nonumber
\leq~ & \left\vert h^d \sum_{\xi\in \Kc_h^W} g\big(\hH(\xi)^{\DD_{W,L}} \big)_{\bzero,\bzero} - h^d \sum_{\xi\in \Kc_h} g\big( \hH(\xi)^{\DD_{W,L}}\big)_{\bzero,\bzero}\right\vert 
\\[1ex] \nonumber
& + \left \vert h^d \sum_{\xi\in \Kc_h} g\big( \hH(\xi)^{\DD_{W,L}}\big)_{\bzero,\bzero} - \int_{\R^d} g\big(\hH(\xi)^{\DD_{W,L}} \big)_{\bzero,\bzero} \dd \xi \right\vert 
\\[1ex] \nonumber
& + \left\vert \int_{\R^d} g\big(\hH(\xi)^{\DD_{W,L}} \big)_{\bzero,\bzero} \dd \xi - \int_{\R^d} g\big(\hH(\xi)\big)_{\bzero,\bzero} \dd \xi \right\vert 
\\[1ex] \nonumber
=: ~& T_1 + T_2 + T_3.
\end{align}

To estimate the first term $T_1$, we have from \eqref{decay-gH-xi-ext} that 
\begin{equation*}
\big[g\big(\hH(\xi)^{\DD_{W,L}} \big)\big]_{\bzero,\bzero} \leq C \delta^{-2} e^{- c \zeta |\xi|}, 
\end{equation*}
which implies
\begin{align}
\label{proof-4-3-T1}
T_1 = h^d \sum_{\xi\in\Kc_h\backslash\Kc_h^W}\big[g\big(\hH(\xi)^{\DD_{W,L}} \big)\big]_{\bzero,\bzero}
\leq C \delta^{-2} \int_{\R^d\backslash[W,W]^d} \big[g\big(\hH(\xi)^{\DD_{W,L}} \big)\big]_{\bzero,\bzero} \dd\xi \leq C \delta^{-2} e^{- c \zeta W} .
\end{align}

The second term $T_2$ represents the quadrature error of trapezoidal rule.
We observe that for $g\in\setg$, $\big[g\big(\hH(\xi)^{\DD_{W,L}} \big)\big]_{\bzero,\bzero}$ is analytic in the region $\big\{\xi=(\xi_1,\cdots,\xi_d) \in\C^d ~:~ |{\rm Im}(\xi_i)|\leq a\delta, ~i=1,\cdots,d \big\}$ with some constant $a>0$.
Then by using \cite[Theorem 5.1]{trefethen2014exponentially}, we have that the quadrature error of trapezoidal rule decays exponentially fast with respect to the mesh size, i.e. there exists a constant $c_a>0$ depending on $a$ such that
\begin{equation}
\label{proof-4-3-T2}
T_2 \leq C e^{-c_a\delta/h}. 
\end{equation}

To estimate the last term $T_3$, we shall combine the estimate \eqref{proof-decay-gH-xi} with that in \eqref{diffldosW} and derive that 
\begin{equation}
\label{decay-gHxi-diff-ext}
\left\vert \lodos{\hH(\xi)^{\DD_{W,L}}}\bzero -\lodos{\hH(\xi)}\bzero\right\vert \leq C  \delta^{-2}  
e^{- c\zeta |\xi|} \big( e^{- c\delta L} + e^{- c \zeta W}\big).
\end{equation}
The combination is direct as the two estimates come from two completely different parts in the contour integral: \eqref{diffldosW} is from the difference between the resolvents and \eqref{proof-decay-gH-xi} is from the decay of $g$ respectively.
We then have from the estimate \eqref{decay-gHxi-diff-ext} that
\begin{eqnarray}
\label{proof-4-3-T3}
T_3 \leq C \delta^{-2} \big( e^{- c\delta L} + e^{- c \zeta W}\big) \int_{\R^d}  e^{- c\zeta |\xi|} \dd\xi
\leq C \delta^{-2} \big( e^{- c\delta L} + e^{- c\zeta W} \big). 
\end{eqnarray}

Taking into accounts \eqref{proof-4-3-a}, \eqref{proof-4-3-T1}, \eqref{proof-4-3-T2} and \eqref{proof-4-3-T3}, we have 
\begin{equation*}
\left\vert\DosK{W,L}{\DD_{W,L}} - \aTr\big(g(H)\big) \right\vert \leq C \Big( \delta^{-2} e^{- c\delta L} + \delta^{-2} e^{- c \zeta W} + e^{- c\delta/h} \Big),
\end{equation*}
which completes the proof.
\end{proof}

\subsection{Proof of Theorem \ref{theorem:WLconvergence-dos}}
\label{sec:prooftheo:WLconvergence-dos}

We recall the statement of Theorem \ref{theorem:WLconvergence-dos}, which gives the error estimate of the numerical approximation \eqref{def:dos-reciprocalWL}. 
Let $g\in \setg$, then there exist positive constants $C$ and $c$ independent of $L$, $W$ and $g$, such that 
\begin{equation*}
\left\vert 
\pwdosWL{\hH^{\DD_{W,L}}} -  \aTr\big(g(H)\big) 
\right\vert 
\leq \pmb{\phi}(L) + C \delta^{-2} e^{-c \zeta W} ,
\end{equation*}
where $\pmb{\phi}(L)\rightarrow 0$ as $L\rightarrow\infty$.
To prove Theorem \ref{theorem:WLconvergence-dos}, we first give the following two lemmas.
We remind that without shifting, $\hat{H}$ is defined by \eqref{Hpw}.

\begin{lemma}
\label{lem:shiftrelation}
Let $g\in \setg$. Then for any $\bG=(G_1,G_2) \in \RL_1^* \times \RL_2^*$, we have
\begin{eqnarray}
\label{shift}
\lodos{\hH}{\bG} = \lodos{\hH(G_1+G_2)}{\bzero}.
\end{eqnarray}
\end{lemma}

\begin{proof}
From the definition \eqref{Hpw} of Hamiltonian matrix elements, we see that for $\xi=G_1+G_2$,
\begin{eqnarray*}
\left(\hH^{\Omega_R(\bG)}\right)_{\bG',\bG''} = \left(\hH(\xi)^{\Omega_R(\bzero)}\right)_{\bG'-\bG,\bG''-\bG}
\qquad\forall~R>0~~ {\rm and}~|\bG'-\bG|, ~ |\bG''-\bG|<R.
\end{eqnarray*}
This implies
\begin{eqnarray}
\label{gHRG}
\left(g\big(\hH^{\Omega_R(\bG)}\big)\right)_{\bG',\bG''} = \left(g\big(\hH(\xi)^{\Omega_R(\bzero)}\big)\right)_{\bG'-\bG,\bG''-\bG} .
\end{eqnarray}
Let $\bG'=\bG''=\bG$ in \eqref{gHRG} and take the limit $R\rightarrow\infty$ on both sides, we can derive \eqref{shift} by using Lemma \ref{lemma:locality:reciprocal}.
\end{proof}

The second lemma mirrors the ergodicity of incommensurate system (see \cite[Proposition 2.4]{cances2017generalized} and \cite[Theorem 2.1 and Page 14]{massatt2017electronic}).
We refer to \cite[Page 13-14]{massatt2017electronic} for the proof. 

\begin{lemma}
\label{lemma:localgeometries}
Let $\RL_1$ and $\RL_2$ be incommensurate lattices, with $\Gamma_1$ and $\Gamma_2$ the associated unit cells respectively. 
Let $k\in\{1,2\}$, $\bar{k}=1$ if $k=2$ and $\bar{k}=2$ if $k=1$.
If $f\in C(\R^d)$ is periodic with respect to $\RL_k$, then 
\begin{equation}
\label{localgeometries}
\lim_{R\to\infty} \frac{1}{\# (\RL_{\bar{k}}\cap B_{R})} \sum_{\ell \in  \RL_{\bar{k}}\cap B_{R} }f(\ell) = \frac{1}{|\Gamma_{k}|} \int_{\Gamma_{k}} f(b) \dd b .
\end{equation}
Moreover, for any given Fourier mode $f(\ell) = e^{i\ell\cdot s}$ with $s\in\RL_k^*$, there exists a constant $C_s>0$ depending on $s$ but independent of $R$, such that
\begin{equation}
\label{error_ergodicity}
\left| \frac{1}{\# (\RL_{\bar{k}}\cap B_{R})} \sum_{\ell \in \RL_{\bar{k}}\cap B_{R} }f(\ell) - \frac{1}{|\Gamma_{k}|} \int_{\Gamma_{k}} f(b) \dd b \right| 
\leq C_s R^{-1} .
\end{equation}
\end{lemma}

Note that the constant $C_s$ in \eqref{error_ergodicity} depends on the Fourier mode $s$, which could explode at some $s\in\RL_k^*$ relying on how the incommensurate lattices $\RL_1$ and $\RL_2$ are misaligned (see \cite[Proposition 2.4]{cances2017generalized} and \cite[Theorem 2.1 and Page 14]{massatt2017electronic}).
For example, this constant could become significantly different for a twisted bilayer graphene as the twist angle varies.

\vskip 0.2cm

\begin{proof}
[\it Proof of Theorem \ref{theorem:WLconvergence-dos}]
We define the following region 
\begin{align}
\label{def:SW}
S_W := \cup_{j\in \RL^*_1 \cap B_W} \big(j + \Gamma_1^*\big) .
\end{align}
By using \eqref{def:dos-reciprocalWL} and $\aTr\big(g(H)\big)
= \int_{\R^d} [g(\hH(\xi))]_{{\bf 0,0}} \dd\xi$ from Theorem \ref{theo:tdllimit}, we can obtain that
\begin{align}
\label{proof-4-1-a} 
& \left\vert \pwdosWL{\hH^{\DD_{W,L}}} -\int_{\R^d} g\big(\hH(\xi)\big)_{\bzero,\bzero} \dd \xi\right\vert 
\\[1ex] \nonumber
\leq~& \frac{|\Gamma_1^*||\Gamma_2^*|}{S_{d,L}}\left\vert\sum_{\bG\in \DD_{W,L}} \lodos{\hH^{\DD_{W,L}}}{\bG}-\sum_{\bG\in\DD_{W,L}} \lodos{\hH}{\bG} \right\vert 
\\[1ex] \nonumber
& + \left\vert \frac{|\Gamma_1^*||\Gamma_2^*|}{S_{d,L}} \sum_{\bG\in\DD_{W, L}}\lodos{\hH}{\bG} - \int_{S_W} \lodos{\hH(\xi)}{\bzero} \dd \xi\right\vert
\\[1ex] \nonumber
& + \left\vert \int_{S_W} \lodos{\hH(\xi)}{\bzero} \dd \xi - \int_{\R^d} \lodos{\hH(\xi)}{\bzero} \dd \xi\right\vert
\\[1ex] \nonumber
=: ~ & T_1+T_2+T_3.
\end{align}

To estimate the first term $T_1$, we have from Lemma \ref{lemma:locality:reciprocal:GG} that for any $\bG=(G_1,G_2) \in \DD_{\frac{W}{2},L}$, there exist a constant $c>0$ independent of $L,W$ and $g$ such that
\begin{align}
\label{proof-4-1-T1-a}
\left\vert \lodos{\hH^{\DD_{W,L}}}{\bG}- \lodos{\hH}{\bG} \right\vert
\leq C \delta^{-2} \Big( e^{-c \delta ( L-|G_1-G_2| )} + e^{-c\zeta( W-|G_1+G_2|)} \Big) .
\end{align}
Moreover, we have from the estimate \eqref{decay-gH-xi-ext} that for any $\bG=(G_1,G_2) \in \DD_{W,L}\backslash\DD_{\frac{W}{2},L}$,
\begin{equation}
\label{proof-4-1-T1-b}
\left\vert \lodos{\hH}{\bG} \right\vert
\leq C \delta^{-2}e^{- c\zeta \frac{W}{2}} 
\quad{\rm and}\quad
\left\vert \lodos{\hH^{\DD_{W,L}}}{\bG} \right\vert \leq C \delta^{-2} e^{- c \zeta \frac{W}{2}} .
\end{equation}
By combing \eqref{proof-4-1-T1-a} and \eqref{proof-4-1-T1-b}, we can obtain
\begin{align}
\label{proof-4-1-T1}
T_1 & \leq \frac{|\Gamma_1^*||\Gamma_2^*|}{S_{d,L}} \Bigg( \sum_{\bG\in \DD_{\frac{W}{2},L}} \left| \lodos{\hH^{\DD_{W,L}}}{\bG} - \lodos{\hH}{\bG} \right|
\\[1ex] \nonumber
& \qquad \qquad \qquad 
+ \sum_{\bG\in \DD_{W,L} \backslash \DD_{\frac{W}{2},L}} \left| \lodos{\hH^{\DD_{W,L}}}{\bG} - \lodos{\hH}{\bG} \right| \Bigg)
\\[1ex] \nonumber
& \leq C \delta^{-2} \left( e^{- c\delta L} + e^{- c\zeta W} \right) .
\end{align}

To estimate the second term $T_2$ in \eqref{proof-4-1-a}, we obtain from the ``symmetry" in Lemma \ref{lem:shiftrelation} and the definition of $S_W$ in \eqref{def:SW} that
\begin{align}
\label{proof-theo4.3-t2}
& \quad T_2 = \bigg| \frac{|\Gamma_1^*||\Gamma_2^*|}{S_{d,L}} \sum_{\bG\in \DD_{W, L}}\lodos{\hH(G_1+G_2)}{\bzero} - \sum_{G_1\in\RL_1^*\cap B_W} \int_{\Gamma_1^*} \lodos{\hH(b+G_1)}\bzero \dd b\bigg|
\\[1ex] \nonumber
& = \Bigg| \frac{|\Gamma_1^*||\Gamma_2^*|}{S_{d,L}} \sum_{\ell\in\RL_2^*\cap B_{\frac{L}{2}}} \sum_{\substack{G_1\in\RL_1^*\\|G_1+\ell|\leq W}} \lodos{\hH(\ell+G_1)}{\bzero} - \int_{\Gamma_1^*} \sum_{G_1\in\RL_1^*\cap B_W}\lodos{\hH(b+G_1)}\bzero \dd b \Bigg|  ,
\end{align}
where the the definition \eqref{set:cutoff} has been used to get the second equality.

Let
\begin{eqnarray}
\label{proof-4-1-T2-a}
f(\ell):=\sum_{G_1\in\RL_1^*,~|G_1+\ell|\leq W} \lodos{\hH(\ell+G_1)}{\bzero} .
\end{eqnarray}
We see that $f(\ell)$ is continuous and periodic with respect to $\RL_2^*$.
Note that $S_{d,L}$ is the volume of the $d$-dimensional ball with diameter $L$, we have the fact $\frac{S_{d,L}}{|\Gamma_2^*|}=\#\big(\RL_2^* \cap B_{\frac{L}{2}}\big)+\mathcal{O}(L^{d-1})$.
Then we can apply Lemma \ref{lemma:localgeometries} with $R=\frac{L}{2}$ in \eqref{localgeometries} and derive that
\begin{align}
\label{proof-4-1-T2-c}
\lim_{L\rightarrow\infty} \bigg| \frac{|\Gamma_1^*||\Gamma_2^*|}{S_{d,L}} \sum_{\ell\in\RL_2^* \cap B_{\frac{L}{2}}} f(\ell)
- \int_{\Gamma_1^*} f(b) \dd b \bigg| = 0 .
\end{align}
Note that although there is an $\mathcal{O}(L^{-1})$ estimate for any given Fourier mode in \eqref{error_ergodicity}, we can not derive such a uniform convergence rate since we do not know $f(\cdot)$ in \eqref{proof-4-1-T2-a} consists of finitely many Fourier modes and the constant $C_s$ in \eqref{error_ergodicity} could explode depending on the incommensurate alignment.

To bridge the gap between \eqref{proof-theo4.3-t2} and \eqref{proof-4-1-T2-c}, we estimate by using \eqref{decay-gH-xi-ext} that
\begin{align}
\label{proof-4-1-T2-e}
\bigg| \sum_{G_1\in\RL_1^*\cap B_W} \lodos{\hH(b+G_1)}{\bzero} - \sum_{G_1\in\RL_1^*,~|G_1+b|\leq W} \lodos{\hH(b+G_1)}{\bzero} \bigg| 
\leq C \delta^{-2} W^{d-1} e^{- c \zeta W}
\end{align}
for any $b\in\Gamma_1^*$.
Then by taking into accounts \eqref{proof-theo4.3-t2}, \eqref{proof-4-1-T2-c} and \eqref{proof-4-1-T2-e}, we obtain
\begin{eqnarray}
\label{proof-4-1-T2}
T_2 \leq \pmb{\phi}(L) + C \delta^{-2} W^{d-1} e^{- c \zeta W} 
\end{eqnarray}
with $\lim_{L\rightarrow\infty}\pmb{\phi}(L) = 0$.

For the last term $T_3$, we have from the estimate \eqref{decay-gH-xi-ext} that
\begin{equation}
\label{proof-4-1-T3}
T_3 \leq C \delta^{-2} \int_{\R^d \backslash S_W} e^{- c\zeta |\xi|} \dd \xi \leq C \delta^{-2} e^{- c \zeta W}.
\end{equation}

Finally, by combing \eqref{proof-4-1-a}, \eqref{proof-4-1-T1}, \eqref{proof-4-1-T2} and \eqref{proof-4-1-T3}, we can complete the proof of \eqref{errordos-reciprocal}.
\end{proof}

\bibliographystyle{siamplain}
\bibliography{bib}

\end{document}

%% file: pre.tex
\usepackage{lipsum}
\usepackage{amsfonts}
\usepackage{graphicx}
\usepackage{epstopdf}
\usepackage{algorithmic}

\usepackage{mathrsfs}
\usepackage{amssymb}
\usepackage{xcolor}
\usepackage{geometry}
\usepackage{url}
\usepackage{setspace}
\usepackage{verbatim}
\usepackage{ulem}
\usepackage{cancel}
\usepackage{multirow}
\usepackage{booktabs}
\usepackage[page,toc,titletoc,title]{appendix}
\usepackage{subfigure}
\usepackage{float}

\ifpdf
  \DeclareGraphicsExtensions{.eps,.pdf,.png,.jpg}
\else
  \DeclareGraphicsExtensions{.eps}
\fi

\newcommand\dd{~{\rm d}}
\newcommand\DD{\mathcal{D}}

\newcommand\R{\mathbb{R}}
\newcommand\N{\mathbb{N}}
\newcommand\Z{\mathbb{Z}}
\newcommand\C{\mathbb{C}}
\newcommand\Cc{\mathscr{C}}

\newcommand\St{\mathcal{S}}

\newcommand\Kc{\mathcal{K}}
\newcommand\RL{\mathcal{R}}

\newcommand\Sc{\mathscr{S}}

\newcommand\df{\mathfrak{d}}

\newcommand\Gl{\mathcal{G}}
\newcommand{\unfold}{\mathcal{T}}
\newcommand{\doss}{{\rm DoS}}
\newcommand{\ldos}{{\rm LDoS}}

\newcommand{\eps}{\varepsilon}

\newcommand{\FT}{\mathcal{F}}

\newcommand\hH{\hat{H}}

\newcommand\Tr{{\rm Tr}}
\newcommand\aTr{\underline{\rm Tr}}
\newcommand\aTrR{\underline{\rm Tr}_R}

\newcommand\bG{{\bf G}}
\newcommand\bzero{{\bf 0}}

\newcommand{\lodos}[2]{\big[g\big(#1\big)\big]_{#2,#2}}
\newcommand{\lodoss}[3]{\big[g\big(#1\big)\big]_{#2,#3}}

\newcommand{\pwdosWL}[1]{\underline{{\widehat{\rm Tr}}}_{W,L}\Big(g\big(#1\big)\Big)}
\newcommand\setg{\Lambda_{\zeta,\delta}}
\newcommand{\DosK}[2]{\underline{\widetilde{\rm Tr}}_{#1}^h\Big( g\big(\hH^{#2}\big)\Big)}

\newsiamremark{remark}{Remark}
\newsiamremark{hypothesis}{Hypothesis}
\crefname{hypothesis}{Hypothesis}{Hypotheses}
\newsiamthm{claim}{Claim}

\headers{Convergence of Planewave Approximations for Incommensurate Systems}{T. Wang, H. Chen, A. Zhou, Y. Zhou and D. Massatt}

\title{Convergence of the Planewave Approximations for Quantum Incommensurate Systems\thanks{Submitted to the editors DATE.
\funding{This work was funded by the National Key R \& D Program of China under grants 2019YFA0709600 and 2019YFA0709601.}
}
}

\author{
Ting Wang\thanks{School of Mathematical Sciences, Beijing Normal University, Beijing 100875, China
(\email{wangting2019@mail.bnu.edu.cn}, \email{chen.huajie@bnu.edu.cn}).}
\and Huajie Chen\footnotemark[2]
\and Aihui Zhou\thanks{
LSEC, Institute of Computational Mathematics and Scientific/Engineering Computing, Academy of Mathematics and Systems Science, Chinese Academy of Sciences, Beijing 100190, China;
School of Mathematical Sciences, University of Chinese Academy of Sciences, Beijing 100049, China
(\email{azhou@lsec.cc.ac.cn}).}
\and Yuzhi Zhou\thanks{
CAEP Software Center for High Performance Numerical Simulation, Beijing 100088, China;
Institute of Applied Physics and Computational Mathematics, Beijing 100094, China
(\email{zhou\_yuzhi@iapcm.ac.cn}).}
\and Daniel Massatt\thanks{
Corresponding author. Mathematics Department, Louisiana State University, LA 70802, USA
(\email{dmassatt@lsu.edu}).}
}

\usepackage{amsopn}
